\documentclass{amsart}

\usepackage{SmoothDeformationsGMConnection}

\title{Smooth Deformations and the Gauss-Manin Connection}
\author{Allan Yashinski}
\thanks{This research was partially supported under NSF grant DMS-1101382.}

\begin{document}

\begin{abstract} Given a smooth one parameter deformation of associative topological algebras, we define Getzler's Gauss-Manin connection on both the periodic cyclic homology and cohomology of the corresponding smooth field of algebras and investigate some basic properties.  We use the Gauss-Manin connection to prove a rigidity result for periodic cyclic cohomology of Banach algebras with finite weak bidimension.
\end{abstract}

\maketitle

\tableofcontents

\section{Introduction}

In this paper, we study the invariance properties of periodic cyclic homology under deformations of the algebra structure.  Given a family of algebras $\{A_t\}_{t \in J}$ parametrized by a real number $t$, we would like to identify conditions under which we can conclude
\[ HP_\bullet(A_t) \cong HP_\bullet(A_s) \qquad \forall t,s \in J. \]
The types of algebras we consider will be topological algebras, and the deformations will have a smooth dependence on $t$.

In the world of formal deformations, Getzler constructed a connection on the periodic cyclic complex of a deformation \cite{MR1261901}.  His connection, called the \emph{Gauss-Manin connection}, commutes with the boundary map on the periodic cyclic complex and descends to a flat connection on the periodic cyclic homology of the deformation.  Our goal is to adapt Getzler's connection to our setting of smooth deformations and investigate its properties.

For a real interval $J \subseteq \R$, we consider a smoothly varying family $\{m_t\}_{t \in J}$ of jointly continuous associative multiplications on a locally convex vector space $X$.  For each $t \in J$, we have a locally convex algebra $A_t$ whose underlying space is $X$ and whose multiplication is given by $m_t$.  These algebras can be collected to form the algebra $A_J$ of smooth sections of the bundle of algebras over $J$ whose fiber at $t \in J$ is $A_t$, where the multiplication in $A_J$ is defined fiberwise.  Then $A_J$ is an algebra over $C^\infty(J)$, the space of smooth complex-valued functions defined on the parameter space $J$, where the module action is given by fiberwise scalar multiplication.



The complex of interest to us is the periodic cyclic complex of $A_J$ over the ground ring $C^\infty(J)$.  This can be thought of as the space of smooth sections of the bundle of chain complexes over $J$ whose fiber at $t \in J$ is the periodic cyclic complex of $A_t$.  It is on this complex that we shall define Getzler's Gauss-Manin connection $\nabla_{GM}$.  The connection $\nabla_{GM}$ commutes with the boundary map and thus descends to a connection on the $C^\infty(J)$-linear periodic cyclic homology $HP_\bullet(A_J)$.

The fundamental issue for us is to determine when we can parallel tranport with respect to $\nabla_{GM}$ at the level of periodic cyclic homology.  Indeed, doing so would provide isomorphisms $HP_\bullet(A_t) \cong HP_\bullet(A_s)$ between the periodic cyclic homology groups of any two algebras in the deformation.  This can be used as a computational device if one already knows the cyclic homology of one particular algebra $A_{t_0}$ in the deformation.  Of course, the striking degree of generality for which $\nabla_{GM}$ exists is an indication that any attempt to integrate $\nabla_{GM}$ will fail generally.  The goal then is to identify properties of a deformation that allow for parallel translation.

Our main result is a rigidity result for periodic cyclic cohomology of a certain class of Banach algebras.  The weak bidimension $\dbw A$ of a Banach algebra $A$ is the smallest integer $n$ such that the Hochschild cohomology $H^{n+1}(A, M^*)$ vanishes for all Banach $A$-bimodules $M$.  A Banach algebra $A$ is called amenable if $\dbw A = 0$.  This class was defined and studied by Johnson \cite{MR0374934}.  If $\dbw A = n$, then $A$ is also called $(n+1)$-amenable.  We prove that the Gauss-Manin connection is integrable for small enough deformations of a Banach algebra of finite weak bidimension.  Consequently, periodic cyclic cohomology is preserved under such deformations.


A general feature of $\nabla_{GM}$ is the fact that
\[ \nabla_{GM}[\ch P] = 0, \]
where $[\ch P]$ denotes the class in $HP_0(A_J)$ of the Chern character of an idempotent $P$ in the algebra $M_N(A_J)$ of $N \times N$ matrices over $A_J$.
One can also define a dual Gauss-Manin connection $\nabla^{GM}$ on the periodic cyclic cohomology $HP^\bullet(A_J)$ over $C^\infty(J)$.  The connections are compatible in the sense that for $[\phi] \in HP^\bullet(A_J)$ and $[\omega] \in HP_\bullet(A_J)$,
\[ \frac{d}{dt} \langle \phi, \omega \rangle = \langle \nabla^{GM}\phi, \omega \rangle + \langle \phi, \nabla_{GM} \omega \rangle, \] where
\[ \langle \cdot, \cdot \rangle: HP^\bullet(A_J) \times HP_\bullet(A_J) \to C^\infty(J) \]
is the canonical pairing.  Combined with the above result, this says
\[ \frac{d}{dt} \langle \phi, \ch P \rangle = \langle \nabla^{GM}\phi, \ch P \rangle \]
for an idempotent $P \in M_N(A_J)$.  This offers insight into how the pairing between $K$-theory and cyclic cohomology deforms as the algebra deforms.

In a subsequent paper \cite{Yashinski}, we prove the integrability of the Gauss-Manin connection for the deformation of smooth noncommutative tori.  We also use the compatibility of the Gauss-Manin connection with the Chern character to prove differentiation formulas for the pairings of cyclic cocycles with $K$-theory classes.  Similar work was carried out independently by Yamashita \cite{Yamashita}.

The outline of the paper is as follows.  In \textsection 2, we cover necessary background on locally convex topological vector spaces, Hochschild and cyclic homology, and Getzler's Cartan homotopy formula for the action of Hochschild cochains on the periodic cyclic complex.  In \textsection 3, we lay the foundation for our study of deformations by studying properties of modules of the form $C^\infty(J,X)$ for some locally convex vector space $X$.  Our main techniques for studying deformations use connections and parallel translation, which are discussed in \textsection 4.  In \textsection 5, we define what we mean by a smooth deformation of either algebras or chain complexes, and give a criterion for triviality of these deformations in terms of integrable connections.  We use our methods to prove some known rigidity results in \textsection 6.  In \textsection 7, we define Getzler's Gauss-Manin connection in our setting of smooth deformations, and prove some of its basic properties.  The main theorem on Banach algebras of finite weak bidimension is proved in \textsection 8.

\subsection*{Acknowledgements}
I'd like to thank my thesis advisor, Nigel Higson, for suggesting this line of research and offering useful guidance throughout the project.  I have benefitted from discussions with Erik Guentner and Rufus Willett on the material.  I'd also like to thank Rufus Willett for reading an earlier draft of this document and suggesting ways to improve it.

\section{Preliminaries}

In this section, we shall first quickly review the relevant concepts from the theory of locally convex topological vector spaces.  Then we'll discuss Hochschild and cyclic homology.  In particular, we'll review the action of Hochschild cochains on the cyclic chain complex and the corresponding Cartan homotopy formula.

\subsection{Locally convex algebras and modules}

We shall work in the category $\LCTVS$ of complete, Hausdorff locally convex topological vector spaces over $\C$ and continuous linear maps.  We shall write $X \in \LCTVS$ to mean that $X$ is a complete, Hausdorff locally convex topological vector space. See \cite{MR0225131} for background in the theory of locally convex topological vector spaces.  More details concerning topological tensor products can be found in \cites{MR0075539, MR0225131, MR1093462}.  We shall rapidly review the necessary facts below.

Recall that $X \in \LCTVS$ is a Fr\'{e}chet space if the topology on $X$ is metrizable.  This is equivalent to the topology on $X$ being defined by a countable family of seminorms.  Among the Fr\'{e}chet spaces are the Banach spaces, whose topology is defined by a single norm.

Given $X \in \LCTVS$ and a subspace $Y \subset X$, there is a naturally locally convex topology on the quotient space $X/Y$, which is Hausdorff if and only if $Y$ is closed.  Even if $Y$ is closed, the quotient $X/Y$ may not be complete.  However if $X$ is a Fr\'{e}chet space and $Y$ is a closed subspace, then $X/Y$ is complete and therefore is also a Fr\'{e}chet space.

Given $X, Y \in \LCTVS$, the space $\Hom(X,Y)$ of continuous linear maps from $X$ to $Y$ is a Hausdorff locally convex topological vector space under the topology of uniform convergence on bounded subsets of $X$.  Recall that a subset $A \subset X$ is bounded if and only if $\sup_{x \in A} p(x) < \infty$ for each continuous seminorm $p$ on $X$.  If $X$ has the additional property of being bornological, then $\Hom(X,Y)$ is complete.  Examples of bornological spaces include Fr\'{e}chet spaces and $LF$-spaces, which are strict, countable inductive limits of Fr\'{e}chet spaces.  A special case of interest is the \emph{strong dual} $X^* = \Hom(X, \C).$  We remark that the strong dual of a Banach space is a Banach space, but the strong dual of a Fr\'{e}chet space is never a Fr\'{e}chet space, unless the original space is actually a Banach space.  One can also consider $\Hom(X,Y)$ with the weaker topology of pointwise convergence, which we denote $\Hom_\sigma(X,Y)$.

Nuclearity (in the sense of Grothendieck) is a nice technical property that a space $X \in \LCTVS$ can have.  We shall occasionally need to reference it, but we shall not work with the concept directly.  For more details, see \cite{MR0075539} or \cite{MR0225131}.  We remark that this should not be confused with the notion of nuclearity for $C^*$-algebras.  For example, a Banach space is nuclear if and only if it is finite-dimensional.

The bilinear maps appearing in structures in this paper will be assumed to be jointly continuous.  This naturally leads to the projective tensor product of two spaces in $\LCTVS$.  If one wishes to consider separately continuous bilinear maps, one should use the inductive tensor product instead, see \cite{MR0075539}.  In the cases of Fr\'{e}chet spaces, the notions of joint and separate continuity coincide, and therefore so do these two tensor products.

Given $X, Y \in \LCTVS$, the \emph{projective topology} on $X \otimes Y$ is the strongest locally convex topology such that the canonical bilinear map $\iota: X \times Y \to X \otimes Y$ is jointly continuous, see \cite[Chapter~43]{MR0225131} for more details and an explicit construction.  The \emph{(completed) projective tensor product} $X \potimes Y$ is the completion of $X \otimes Y$ with the projective topology.  The completed projective tensor product has the universal property that any jointly continuous bilinear map $B$ from $X \times Y$ into a space $Z \in \LCTVS$ induces a unique continuous linear map $\widehat{B}: X \potimes Y \to Z$ such that the diagram
\[ \xymatrix{
X \times Y \ar[r]^-\iota \ar[rd]^-B & X \potimes Y \ar[d]^-{\widehat{B}}\\
& Z
} \]
commutes.  The projective tensor product is functorial in the sense that two continuous linear maps $F: X_1 \to X_2$ and $G: Y_1 \to Y_2$ induce a continuous linear map \[ F \otimes G: X_1 \potimes Y_1 \to X_2 \potimes Y_2 \] given on elementary tensors by
\[ (F \otimes G)(x \otimes y) = F(x) \otimes G(y). \]  If $X$ and $Y$ are Banach (resp. Fr\'{e}chet) spaces, then $X \potimes Y$ is a Banach (resp. Fr\'{e}chet) space.

In the language of category theory, the tensor product $\potimes$ makes $\LCTVS$ into a symmetric monoidal category.  In any such category, one can define a notion of algebra and module.  For $\LCTVS$, these notions agree with the definitions below.

By a \emph{locally convex algebra},
we mean a space $A \in \LCTVS$ equipped with a jointly continuous associative multiplication.  Notice we are implicitly assuming $A$ is complete.  The multiplication induces a continuous linear map $m: A \potimes A \to A$.  Joint continuity implies that for every defining seminorm $p$ on $A$, there is another continuous seminorm $q$ such that
\[ p(ab) \leq q(a)q(b), \qquad \forall a,b \in A. \]  There may be no relationship between $p$ and $q$ in general.  In the special case where
\[ p(ab) \leq p(a)p(b), \qquad \forall a,b \in A, \] for all seminorms in a family defining the topology, the algebra is called \emph{multiplicatively convex} or \emph{$m$-convex}.  An $m$-convex algebra can be expressed as a projective limit of Banach algebras.  A \emph{Fr\'{e}chet algebra} is a locally convex algebra whose underlying space is a Fr\'{e}chet space.  We do not insist that a Fr\'{e}chet algebra is $m$-convex, as some authors do.

Now suppose $R$ is a unital, commutative locally convex algebra.  By a \emph{locally convex $R$-module}, we mean a space $M \in \LCTVS$ equipped with a jointly continuous $R$-module structure.  Such a module action induces a continuous linear map $\mu: R \potimes M \to M$.  All such modules will be assumed to be unital in the sense that $1 \cdot m = m$ for all $m \in M$.
Given two locally convex $R$-modules $M$ and $N$, we topologize $\Hom_R(M,N)$ as a subspace of $\Hom_\C(M,N)$.  When $N = R$, we obtain the topological $R$-linear dual of $M$
\[ M^\dual := \Hom_R(M,R). \]
We shall use the notation $M^\dual$ to distinguish from $M^*$, which will always mean the usual $\C$-linear topological dual space.

We shall need to take topological tensor products over an algebra different from $\C$.  The basic facts we need are below, but a more detailed exposition can be found in \cite[Chapter II]{MR1093462}.  Suppose $R$ is a unital commutative locally convex algebra and let $M$ and $N$ be locally convex $R$-modules.
The \emph{(completed) projective tensor product over $R$} of $M$ and $N$, denoted $M \potimes_R N$, is a locally convex $R$-module together with a jointly continuous $R$-bilinear map $\iota: M \times N \to M \potimes_R N$ which is universal in the sense that any jointly continuous $R$-bilinear map $B$ from $M \times N$ into a locally convex $R$-module $P$ induces a unique continuous $R$-linear map $\widehat{B}: M \potimes_R N \to P$ such that the diagram
\[ \xymatrix{
M \times N \ar[r]^-\iota \ar[rd]^-B & M \potimes_R N \ar[d]^-{\widehat{B}}\\
& P
} \]
commutes.  The module $M \potimes_R N$ can be constructed as the completion
of $(M \potimes_\C N)/K$, where $K$ is the closure of the subspace spanned by elements of the form
\[ (r\cdot m) \otimes n - m \otimes (r\cdot n), \qquad r \in R, m \in M, n \in N. \]  Any two continuous $R$-linear maps
$F: M_1 \to N_1$ and $G: M_2 \to N_2$ induce a continuous $R$-linear map
\[ F \otimes G: M_1 \potimes_R N_1 \to M_2 \potimes_R N_2 \] in the usual way.  From the construction, we see that if $M$ and $N$ are Banach (resp. Fr\'{e}chet) modules, then $M \potimes_R N$ is a Banach (resp. Fr\'{e}chet) module.


A \emph{locally convex $R$-algebra} is a locally convex $R$-module $A$ equipped with a jointly continuous associative $R$-linear product $m$, so that $m$ induces a continuous $R$-linear map
\[ m: A \potimes_R A \to A. \]


A locally convex $R$-module is \emph{free} if it is isomorphic to $R \potimes_\C X$ for some $X \in \LCTVS$.  Here the $R$-module action is induced by
\[ r\cdot(s \otimes x) = rs \otimes x. \]
The free module $R \potimes X$ has the universal property that any continuous $\C$-linear map $\widetilde{F}$ from $X$ into a locally convex $R$-module $M$ induces a unique continuous $R$-linear map $F: R \potimes X \to M$ such that the diagram
\[ \xymatrix{
X \ar[r]^-\iota \ar[rd]^-{\widetilde{F}} & R \potimes X \ar[d]^-F\\
& M
} \]
commutes.  Here, $\iota(x) = 1 \otimes x$ and $F$ is given by
\[ F(r \otimes x) = r\cdot \widetilde{F}(x). \]
This establishes a linear isomorphism
\[ \Hom_R(R \potimes X, M) \cong \Hom(X, M). \]

\begin{proposition} \label{Proposition-FreeModuleUMPTopologicalIsomorphism}
If $X$ and $R$ are Fr\'{e}chet spaces, one of which is nuclear, then the linear isomorphism \[ \Hom_R(R \potimes X, M) \cong \Hom(X, M) \] is a topological isomorphism.
\end{proposition}

\begin{proof}
The linear isomorphisms
\[ \Phi: \Hom(X,M) \to \Hom_R(R \potimes X, M), \qquad \Psi: \Hom_R(R \potimes X, M) \to \Hom(X,M) \] are given by
\[ \Phi(F) = \mu (1 \otimes F), \qquad \Psi(G)(x) = G(1 \otimes x), \]  where $\mu: R \potimes M \to M$ is the module action.  Continuity of $\Psi$ follows from the fact that if $B \subset X$ is bounded, then $1 \otimes B \subset R \potimes X$ is bounded.

The continuity of $\Phi$ is more subtle.  The map $\Phi$ factors as
\[ \xymatrix{
\Phi: \Hom(X,M) \ar[r]^-{\Phi_1} &\Hom_R(R \potimes X, R \potimes M) \ar[r]^-{\Phi_2} &\Hom_R(R \potimes X, M),
} \]
where $\Phi_1(F) = 1 \otimes F$ and $\Phi_2$ is composition with the module action $\mu$.  Continuity of $\Phi_2$ follows from continuity of $\mu$.  To show that $\Phi_1$ is continuous, we need to relate the bounded subsets of $R \potimes X$ to the bounded subsets of $R$ and $X$.  This is related to the difficult ``probl\`{e}me des topologies" of Grothendieck \cite{MR0075539}.  If either $R$ or $X$ are nuclear, then for every bounded subset $D \subset R \potimes X$, there are bounded subsets $A \subset R$, $B \subset X$ such that $D$ is contained in the closed convex hull of
\[ A \otimes B = \{ r \otimes x ~|~ r \in A, x \in B \}, \]  see \cite[Theorem 21.5.8]{MR632257}.  Continuity of $\Phi_1$ follows from this fact.
\end{proof}

\begin{proposition} \label{Proposition-TensorProductOfFreeModules}
Given $X, Y \in \LCTVS$,
\[ (R \potimes_\C X) \potimes_R (R \potimes_\C Y) \cong R \potimes_\C (X \potimes_\C Y) \]
as locally convex $R$-modules via the correspondence
\[ (r_1 \otimes x) \otimes (r_2 \otimes y) \longleftrightarrow r_1r_2 \otimes (x \otimes y). \]
\end{proposition}
This can be proved using the universal properties of both modules.  This shows that the projective tensor product of free modules is free.

\begin{proposition} \label{Proposition-QuotientOfFreeModules}
Let $R$ be a nuclear Fr\'{e}chet algebra.  Given a Fr\'{e}chet space $X$ and a closed subspace $Y \subset X$,
\[ (R \potimes X) / (R \potimes Y) \cong R \potimes (X / Y) \]
as Fr\'{e}chet $R$-modules via the correspondence
\[ [r \otimes x] \longleftrightarrow r \otimes [x]. \]
\end{proposition}

\begin{proof}
Nuclearity of $R$ implies that $R \potimes Y$ is a closed subspace of $R \potimes X$ \cite[Proposition 43.7]{MR0225131}.  Since all spaces are Fr\'{e}chet, all quotients appearing are complete.  One then induces mutually inverse isomorphisms using the universal properties of completed projective tensor products and quotients.
\end{proof}


By a \emph{locally convex cochain complex}, we mean a collection of spaces $\{ \mathcal{C}^n \}_{n \in \Z}$ in $\LCTVS$ and continuous linear maps $\{ d^n: \mathcal{C}^n \to \mathcal{C}^{n+1}\}_{n \in \Z}$ such that $d^{n+1} \circ d^n = 0$.  We'll use the notation $Z^n(\mathcal{C}) = \ker d^n$ for cocycles and $B^n(\mathcal{C}) = \im d^{n-1}$ for coboundaries.  The cohomology is $H^n(\mathcal{C}) = Z^n(\mathcal{C}) / B^n(\mathcal{C})$, which may not be Hausdorff or complete.  We will often drop the superscript $n$ on the coboundary map.  By turning the arrows around, we obtain the definition of a locally convex chain complex.

If each $\mathcal{C}^n$ is a locally convex $R$-module and the coboundary maps are $R$-linear, then $\mathcal{C}^\bullet$ is a \emph{locally convex cochain complex of $R$-modules}.  In this case, the cohomology  spaces are $R$-modules.

\subsection{Hochschild and cyclic homology for locally convex algebras}

A good reference for Hochschild and cyclic homology is \cite{MR1600246}.

Let $R$ be a unital commutative locally convex algebra and let $A$ be a (possibly nonunital) locally convex $R$-algebra.  The main examples for us will be $R = \C$ and $R=C^\infty(J)$, the smooth functions on a real interval $J$.  All homology theories that follow are the continuous versions of the usual $R$-linear algebraic theories, in that they take into account the topology of the algebra $A$.

Recall that the \emph{unitization} of the algebra $A$ is the algebra
\[ A_+ = A \oplus R \] with multiplication
\[ (a_1, r_1)(a_2, r_2) = (a_1a_2 + r_2\cdot a_1 + r_1\cdot a_2, r_1r_2). \]  Then $A_+$ is a unital locally convex $R$-algebra with unit $(0,1)$, which contains $A$ as a closed ideal.  We can, and will, form the unitization in the case where $A$ is already unital.  We shall let $e \in A_+$ denote the unit of $A_+$, to avoid possible confusion with the original unit of $A$, if it exists.

\subsubsection{Hochschild cochains}
Let $C^k(A,A)$ denote the space of all jointly continuous $k$-multilinear (over $R$) maps $D: A^{\times k} \to A$.
The coboundary map $\delta: C^k(A,A) \to C^{k+1}(A,A)$ is given by
\begin{align*}
\delta D(a_1, \ldots, a_{k+1}) &= D(a_1, \ldots , a_k)a_{k+1} + (-1)^{k+1} a_1D(a_2, \ldots, a_{k+1})\\ & \qquad + \sum_{j=1}^k (-1)^{k-j+1}D(a_1, \ldots , a_{j-1}, a_ja_{j+1}, a_{j+2}, \ldots , a_k),
\end{align*} and satisfies $\delta^2 = 0$.
The cohomology of $(C^{\bullet}(A,A), \delta)$ is the \emph{Hochschild cohomology of $A$ (with coefficients in $A$)}, and is denoted by $H^\bullet(A,A)$.  If we wish to emphasize the ground ring $R$, we shall write $H^\bullet_R(A,A).$

There is much additional structure on $C^\bullet(A,A)$, including a cup product and a Lie bracket, called the \emph{Gerstenhaber bracket}.  The shifted complex $\g^\bullet(A) := C^{\bullet+1}(A,A)$ is a differential graded Lie algebra under the Gerstenhaber bracket \cite{MR0161898}.  This gives the cohomology $H^{\bullet+1}(A,A)$ the structure of a graded Lie algebra.

\subsubsection{Hochschild homology}
For $n \geq 0$, the space of \emph{Hochschild $n$-chains} is defined to be
\[ 
C_n(A) = \begin{cases} A, &n=0\\ A_+ \potimes_R A^{\potimes_R n}, &n \geq 1
\end{cases} \]
The boundary map $b: C_n(A) \to C_{n-1}(A)$ is given on elementary tensors by
\begin{align*}
b(a_0 \otimes \ldots \otimes a_n) &= \sum_{j=0}^{n-1} (-1)^j a_0 \otimes \ldots \otimes a_{j-1} \otimes a_ja_{j+1} \otimes a_{j+2} \otimes \ldots \otimes a_n\\
& \qquad + (-1)^n a_na_0 \otimes a_1 \otimes \ldots \otimes a_{n-1}.\\
\end{align*}
More rigorously, $b$ is induced by the functoriality of the projective tensor product $\potimes_R$ using the continuous multiplication map $m: A \potimes_R A \to A$.  This shows that $b$ is continuous.  Associativity of $m$ implies that $b^2 = 0$.  The homology of the complex $(C_\bullet(A), b)$ is called the \emph{Hochschild homology of $A$ (with coefficients in $A_+$)} and shall be denoted $HH_\bullet(A)$ or $HH_\bullet^R(A)$ if we wish to emphasize $R$.

\subsubsection{Cyclic homology}

We only introduce the periodic cyclic theory.  Let \[ C_{\even}(A) = \prod_{n=0}^\infty C_{2n}(A), \qquad C_{\odd}(A) = \prod_{n=0}^\infty C_{2n+1}(A), \] with the product topologies.  Consider the operator $B: C_n(A) \to C_{n+1}(A)$ given on elementary tensors by \[ B(a_0 \otimes \ldots \otimes a_n) = \sum_{j=0}^n (-1)^{jn} e \otimes a_j \otimes \ldots a_n \otimes a_0 \otimes \ldots \otimes a_{j-1} \]  if $a_0 \in A$, and
\[ B(e \otimes a_1 \otimes \ldots \otimes a_n) = 0. \]
Then it is immediate that $B^2 = 0$.  Moreover, one can check that \[ bB + Bb = 0.\]  Extend the operators $b$ and $B$ to the periodic cyclic complex \[ C_{\per}(A) = C_{\even}(A) \oplus C_{\odd}(A). \]  This is a $\Z/2$-graded complex
\[ \xymatrix{
C_{\even}(A) \ar@<.5ex>[r]^-{b+B} &C_{\odd}(A) \ar@<.5ex>[l]^-{b+B}
} \]
with differential $b+B$.  The homology groups of this complex are called the even and odd \emph{periodic cyclic homology groups} of $A$, and are denoted $HP_0(A)$ and $HP_1(A)$ respectively.  As before, we will write $HP_\bullet^R(A)$ if we wish to emphasize the ground ring $R$.

\subsubsection{Dual cohomology theories}

To obtain periodic cyclic cohomology, we dualize the previous notions.  Let \[ C^n(A) = C_n(A)^\dual = \Hom_R(C_n(A),R)\] be the topological $R$-linear dual module of $C_n(A)$ with the topology of uniform convergence on bounded subsets.  The maps \[ b:C^n(A) \to C^{n+1}(A), \qquad B: C^n(A) \to C^{n-1}(A) \] are induced by duality, and are given explicitly by
\begin{align*}
b\phi(a_0, \ldots, a_n) &= \sum_{j=0}^{n-1} (-1)^j \phi(a_0, \ldots a_{j-1}, a_ja_{j+1}, a_{j+2}, \ldots, a_n)\\
& \qquad \qquad + (-1)^n\phi(a_na_0, a_1, \ldots, a_{n-1}),
\end{align*} and
\[ B\phi(a_0, \ldots , a_{n-1}) = \sum_{j=0}^{n-1} (-1)^{j(n-1)}\phi(e, a_j, \ldots, a_{n-1}, a_0, \ldots a_{j-1}), \qquad a_0 \in A, \]
\[ B\phi(e, a_1, \ldots , a_{n-1}) = 0. \]  The cohomology of $(C^\bullet(A), b)$ is called the \emph{Hochschild cohomology of $A$ (with coefficients in $A^\dual = \Hom_R(A, R)$)} and will be denoted by $HH^\bullet(A)$.  The periodic cyclic cochain complex is $C^{\per}(A) = C^{\even}(A) \oplus C^{\odd}(A)$, where
\[C^{\even}(A) = \bigoplus_{n=0}^\infty C^{2n}(A), \qquad C^{\odd}(A) = \bigoplus_{n=0}^\infty C^{2n+1}(A). \]  Then $C^{\per}(A)$ is a $\Z/2$-graded complex with differential $b+B$, and its cohomology groups are the even and odd \emph{periodic cyclic cohomology} of $A$, denoted $HP^0(A)$ and $HP^1(A)$ respectively.

Since $C^{\per}(A) \cong C_{\per}(A)^\dual$, there is a canonical pairing \[ \langle \cdot, \cdot \rangle: C^{\per}(A) \times C_{\per}(A) \to R \]  which descends to a bilinear map \[ \langle \cdot , \cdot \rangle: HP^{\bullet}(A) \times HP_{\bullet}(A) \to R.\]

\subsubsection{Chern character}

We shall review some basic facts about the Chern character in periodic cyclic homology, see \cite[Ch. 8]{MR1600246} for a more detailed account.

Let $A$ be an arbitrary algebra over the ground ring $R$.  Given an idempotent $P \in A$, $P^2 = P$, define the element $\ch P \in C_{\even}(A)$ given by $(\ch P)_0 = P$ and for $n \geq 1$, \[ (\ch P)_{2n} = (-1)^n\frac{(2n)!}{n!}(P^{\otimes (2n+1)} - \frac{1}{2} e \otimes P^{\otimes (2n)}). \]  One can verify directly that $b (\ch P_{2(n+1)}) = -B (\ch P_{2n})$, so that $(b+B)\ch P = 0$.

More generally, we can define $\ch P \in C_{\even}(A)$ when $P$ is an idempotent in the matrix algebra $M_N(A) \cong M_N(\C) \otimes A$.  Consider the \emph{generalized trace} $T: C_\bullet(M_N(A)) \to C_\bullet(A)$ defined by
\[ T((u_0 \otimes a_0) \otimes \ldots \otimes (u_n \otimes a_n)) = \tr(u_0\ldots u_n)a_0 \otimes \ldots \otimes a_n, \] where $\tr: M_N(\C) \to \C$ is the ordinary trace.  As shown in \cite[Ch. 1]{MR1600246}, $T$ is a chain homotopy equivalence, and so induces an isomorphism $HP_\bullet(M_N(A)) \cong HP_\bullet(A)$.  So define $\ch P \in C_{\even}(A)$ to be the image of $\ch P \in C_{\even}(M_N(A))$ under the map $T$.  In this way, we build a homomorphism
\[ \ch: K_0(A) \to HP_0(A), \qquad \ch [P] = [\ch P], \]
where $K_0(A)$ denotes the algebraic $K$-theory group of $A$, and $[P]$ is the $K$-theory class of an idempotent $P \in M_N(A)$.

Given an invertible $U \in A$, there is a cycle $\ch U \in C_{\odd}(A)$ given by \[ (\ch U)_{2n+1} = (-1)^n n! U^{-1} \otimes U \otimes U^{-1} \otimes \ldots \otimes U^{-1} \otimes U.\]  Then, one can check that $(b+B)\ch U = 0$.  As in the case of idempotents, define $\ch U \in C_{\odd}(A)$ for any invertible $U \in M_N(A)$ by composing with $T$.  In this way, we build a homomorphism
\[ \ch: K_1(A) \to HP_1(A), \qquad \ch [U] = [\ch U], \] where $K_1(A)$ denotes the algebraic $K$-theory group of $A$.

There are pairings
\[ HP^0(A) \times K_0(A) \to R, \qquad HP^1(A) \times K_1(A) \to R \] given by
\[ \langle [\phi], [P] \rangle = \langle [\phi], [\ch P] \rangle, \qquad \langle [\phi], [U] \rangle = \langle [\phi], [\ch U] \rangle \] for an idempotent $P \in M_N(A)$ and an invertible $U \in M_N(A)$.

\subsubsection{Noncommutative geometry dictionary}

In the case where $A = C^\infty(M)$, the algebra of smooth functions on a closed manifold $M$ with its usual Fr\'{e}chet topology, the above homology groups have geometric interpretations.  The Hochschild cohomology $H^\bullet(A,A)$ is the graded space of multivector fields on $M$.  The cup product corresponds to the wedge product of multivector fields, and the Gerstenhaber bracket corresponds to the Schouten-Nijenhuis bracket.  The Hochschild homology $HH_\bullet(A)$ is the space of differential forms on $M$.  The differential $B$ descends to a differential on $HH_\bullet(A)$, and this can be identified with the de Rham differential $d$ up to a constant.  The even (respectively odd) periodic cyclic homology can be identified with the direct sum of the even (respectively odd) de Rham cohomology groups.  In a dual fashion, the Hochschild cohomology $HH^\bullet(A)$ is the space of de Rham currents and the periodic cyclic cohomology can be identified with de Rham homology.  For more details, see \cite{MR823176}.

When passing to an arbitrary, not necessarily commutative, algebra $A$, we could view $H^\bullet(A,A)$ and $HH_\bullet(A)$ as spaces of noncommutative multivector fields and differential forms respectively.  However, these spaces have a tendency to be badly behaved.  For example, they may be too small or non-Hausdorff.  Instead, we shall work with the chain complexes $C^\bullet(A,A)$ and $C_\bullet(A)$, and view their elements as (generalized) noncommutative multivector fields and differential forms respectively.
Just as multivector fields act on differential forms by Lie derivative and contraction operations, there are Lie derivative and contraction operations
\[ L, \iota: C^\bullet(A,A) \to \End(C_\bullet(A)) \] for any algebra $A$, which we shall review in the next section.

\subsection{Operations on the cyclic complex}

The Cartan homotopy formula that follows was first observed by Rinehart in \cite{MR0154906} in the case where $D$ is a derivation, and later in full generality by Getzler in \cite{MR1261901}, see also \cite{MR2308582}, \cite{MR1667686}.  An elegant and conceptual proof of the Cartan homotopy formula can be found in \cite{MR1468938}.  Our conventions vary slightly from \cite{MR1261901}, and are like those in \cite{MR2308582}.

To simplify the notation of what follows, the elementary tensor $a_0 \otimes a_1 \otimes \ldots \otimes a_n \in C_n(A)$ will be written as $(a_0, a_1, \ldots , a_n)$.  All operators that are defined in this section are given algebraically on elementary tensors, and extend to continuous linear operators on the corresponding projective tensor products.

All commutators of operators that follow are graded commutators.  That is, if $S$ and $T$ are homogenous operators of degree $|S|$ and $|T|$, then
\[ [S, T] = ST - (-1)^{|S||T|}TS. \]

\subsubsection{Lie derivatives, contractions, and the Cartan homotopy formula}

Given a Hochschild cochain $D \in C^k(A,A)$, the \emph{Lie derivative along $D$} is the operator $L_D \in \End(C_{\bullet}(A))$  of degree $1 - k$ given by
\begin{align*}
&L_D(a_0 , \ldots , a_n)\\
&\qquad = \sum_{i=0}^{n-k+1}(-1)^{i(k-1)}(a_0, \ldots , D(a_i, \ldots , a_{i+k-1}), \ldots , a_n)\\
& \qquad \qquad + \sum_{i=1}^{k-1}(-1)^{in}(D(a_{n-i+1}, \ldots , a_n, a_0, \ldots , a_{k-1-i}), a_{k-i}, \ldots , a_{n-i}).
\end{align*}
In the case $D \in C^1(A,A)$, the above formula is just \[ L_D(a_0 , \ldots , a_n) = \sum_{i=0}^n (a_0 , \ldots a_{i-1} , D(a_i) , a_{i+1} , \ldots , a_n). \]  To be completely precise in the above formulas, we are identifying $C^k(A,A)$ as a subspace of $\Hom((A_+)^{\potimes_R k}, A)$ by extending by zero, so that
\[ D(a_1, \ldots , a_k) = 0, \qquad \text{ if } a_i = e \text{ for some } i. \]
The one exception, where we do not wish to extend by zero, is for the multiplication map $m$ of the unitization $A_+$.  Here, the formula for $L_m$ still gives a well-defined operator on $C_\bullet(A)$, and $L_m = b$.

\begin{proposition} \label{Proposition-LieDerivativeCommutator}
If $D, E \in C^{\bullet}(A,A)$, then 
\[ [L_D, L_E] = L_{[D,E]}, \qquad [b, L_D] = L_{\delta D}, \qquad [B, L_D] = 0. \]
\end{proposition}

So $C_{-\bullet}(A)$ and $C_{\per}(A)$ are differential graded modules over the differential graded Lie algebra $\g^\bullet(A)$.  In particular, the graded Lie algebra $H^{\bullet+1}(A,A)$ acts via Lie derivatives on both the Hochschild homology $HH_{-\bullet}(A)$ and the periodic cyclic homology $HP_{\bullet}(A)$.

Given a $k$-cochain $D \in C^k(A,A)$, the \emph{contraction with $D$} is the operator $\iota_D \in \End(C_\bullet(A))$ of degree $-k$ given by 
\[ \iota_D(a_0 , \ldots , a_n) = (-1)^{k-1}(a_0D(a_1, \ldots , a_k) , a_{k+1} , \ldots , a_n). \]

\begin{proposition} \label{Proposition-ContractionCommutator}
For any $D \in C^\bullet(A,A)$, $[b, \iota_D] = -\iota_{\delta D}$.
\end{proposition}

Although $\iota_D$ interacts well with $b$, it does not with the differential $B$, and needs to be adjusted for the cyclic complex.  Given $D \in C^k(A,A)$, let $S_D$ denote the operator on $C_{\bullet}(A)$ of degree $2-k$ given by
\begin{align*}
&S_D(a_0 , \ldots , a_n) = \sum_{i=1}^{n-k+1}\sum_{j=0}^{n-i-k+1} (-1)^{i(k-1) + j(n-k+1)}\\
& \qquad (e , a_{n-j+1} , \ldots , a_n , a_0 , \ldots , a_{i-1} , D(a_i, \ldots , a_{i+k-1}) , a_{i+k} \ldots , a_{n-j}),
\end{align*}  if $a_0 \in A$ and \[ S_D(e, a_1, \ldots , a_n) = 0. \]  The sum is over all cyclic permutations with $D$ appearing to the right of $a_0$.
Given $D \in C^\bullet(A,A)$, the \emph{cyclic contraction with $D$} is the operator \[ I_D = \iota_D + S_D. \]

\begin{theorem}[Cartan homotopy formula] \label{Theorem-CHF}
For any $D \in C^\bullet(A,A)$, \[ [b+B, I_D] = L_D - I_{\delta D}. \]
\end{theorem}

Theorem \ref{Theorem-CHF} implies that the Lie derivative along a Hochschild cocycle $D \in C^{\bullet}(A,A)$ is continuously chain homotopic to zero in the periodic cyclic complex.  Thus, the action of $H^{\bullet+1}(A,A)$ on $HP_{\bullet}(A)$ by Lie derivatives is zero.

The results of this section can be summarized in another way.  Consider the endomorphism complex $\End_R(C_{\per}(A))$ whose coboundary map is given by the graded commutator with $b + B$.  Let
\[ \Op(A) = \Hom_R\big(\g^\bullet(A), \End_R(C_{\per}(A))\big), \] and let $\partial$ denote the boundary map in $\Op(A)$.  Given $\Phi \in \Op(A)$ and $D \in \g^\bullet(A)$, we shall write $\Phi_D := \Phi(D)$.  So
\[ (\partial \Phi)_D = [b + B, \Phi_D] - (-1)^{|\Phi|}\Phi_{\delta D}. \]  Note that the Lie derivative $L$ and the cyclic contraction $I$ are elements of $\Op(A)$ of even and odd degrees respectively.  Theorem~\ref{Theorem-CHF} is exactly the statement
\[ \partial I = L. \]  So it follows from this that $\partial L = 0$, i.e.
\[ [b+B, L_D] = L_{\delta D}, \]
as in Proposition~\ref{Proposition-LieDerivativeCommutator}.

\begin{example} \label{Example-GradedAlgebraCyclicHomology}
Consider a nonnegatively graded algebra $A = \bigoplus_{n=0}^\infty A_n$.  Let $D: A \to A$ be the algebra derivation defined by $D(a) = n\cdot a$ for all $a \in A_n$.    The complex $C_{\per}(A)$ decomposes into eigenspaces for $L_D$, depending on the total degree of a tensor.  However, $L_D$ acts by zero on $HP_\bullet(A)$ by Theorem \ref{Theorem-CHF}.  Thus the nontrivial part of the homology is contained entirely in the $0$-eigenspace for $L_D$, which coincides with $C_{\per}(A_0)$.  In this way, we see the inclusion $A_0 \to A$ induces an isomorphism $HP_\bullet(A_0) \cong HP_\bullet(A)$.
\end{example}

%
%
%

\section{$C^\infty(J)$-modules}

Let $J \subseteq \R$ be a nonempty open interval, which will serve as a parameter space.  Loosely speaking, our general approach to deformation theory is as follows: given a family of objects $\{E_t\}_{t \in J}$ that depend smoothly on $t$, we form a bundle $E$ over $J$ whose fiber at $t$ is $E_t$.  The object of interest is then the space of smooth sections of the bundle $E$.  If each $E_t$ has an underlying vector space structure (for example, if we are dealing with algebras, chain complexes, Lie algebras, etc), then the space of smooth sections is a $C^\infty(J)$-module.  In what follows, the vector spaces are in $\LCTVS$, so we will deal with locally convex $C^\infty(J)$-modules.  The space of sections will generally inherit new structure by considering any additional structure $\{E_t\}_{t \in J}$ had fiberwise.

Let $X \in \LCTVS$ and consider the space $C^\infty(J,X)$ of infinitely differentiable functions on $J$ with values in $X$.  By a differentiable function $f: J \to X$, we mean that the usual limit
\[ \lim_{h \to 0} \frac{f(t+h) - f(t)}{h} \] exists in the topology on $X$ for all $t \in J$.
We equip $C^\infty(J,X)$ with its usual topology of uniform convergence of functions and all their derivatives on compact subsets of $J$.  We shall write $C^\infty(J) = C^\infty(J, \C)$, which is a nuclear Fr\'{e}chet algebra under this topology.  Notice $C^\infty(J,X)$ is a locally convex module over $C^\infty(J)$ using pointwise scalar multiplication.
Since $X$ is complete,
\[ C^\infty(J,X) \cong C^\infty(J) \potimes X, \] see e.g. \cite[Theorem~44.1]{MR0225131}.  In other words, $C^\infty(J,X)$ is a free locally convex $C^\infty(J)$-module.  If $X$ is a Fr\'{e}chet space, then $C^\infty(J,X)$ is a Fr\'{e}chet space.

The space $C^\infty(J,X)$ is equipped with continuous linear ``evaluation maps" \[ \epsilon_t: C^\infty(J,X) \to X \] for each $t \in J$ given by $\epsilon_t(x) = x(t)$.

There are several notions for what is meant by saying a collection \[\{F_t: X \to Y\}_{t \in J}\] of continuous linear maps depends smoothly on $t$.  We consider one of the strongest.

\begin{definition} \label{Definition-SmoothFamilyOfLinearMaps}
Given $X, Y \in \LCTVS$, a \emph{smooth family of continuous linear maps} from $X$ to $Y$ is a collection of continuous linear maps $\{F_t: X \to Y\}_{t \in J}$ with the property that there exists a continuous linear map $\widetilde{F}: X \to C^\infty(J, Y)$ such that $F_t = \epsilon_t \circ \widetilde{F}$ for all $t \in J$.
\end{definition}

From the universal property of free modules, we see that such a smooth family induces a continuous $C^\infty(J)$-linear map $F: C^\infty(J,X) \to C^\infty(J,Y)$ given by
\[ F(x)(t) = F_t(x(t)), \qquad x \in C^\infty(J,X). \] Conversely, all continuous $C^\infty(J)$-linear maps between free $C^\infty(J)$-modules are induced by a smooth family of continuous linear maps.

Given a smooth family $\{F_t: X \to Y\}_{t \in J}$ of continuous linear maps, it is necessary that the map
\[ t \mapsto F_t(x) \]
is smooth for each $x \in X$.  Under the technical assumption that $X$ is barreled (Fr\'{e}chet spaces are examples of barreled spaces), this is also sufficient.  The Banach-Steinhaus theorem (uniform boundedness principle) \cite[Theorem 33.1]{MR0225131} is the main advantage of considering barreled spaces.

\begin{proposition} \label{Proposition-SmoothFamilyForBarreledX}
If $X$ is barreled, then $\{F_t: X \to Y\}_{t \in J}$ is a smooth family of continuous linear maps if and only if the map
\[ t \mapsto F_t(x) \] is smooth for every $x \in X$.
\end{proposition}

\begin{proof}
Suppose $t \mapsto F_t(x)$ is smooth for each $x \in X$, and let $F_t^{(n)}(x)$ denote the $n$-th derivative of this map.  The linear map $F_t^{(n)}: X \to X$ is in fact continuous for each $t$.  Using induction, this follows from the Banach-Steinhaus theorem because $F_t^{(n)}(x)$ is a pointwise limit of continuous linear maps by its very definition.

We must show that the map $\widetilde{F}: X \to C^\infty(J,Y)$ defined by
\[ \widetilde{F}(x)(t) = F_t(x) \] is continuous.  Our assumption certainly implies that the map $t \mapsto F_t^{(n)}(x)$ is continuous.  Thus, for any compact $K \subset J$, any $n$, and any $x \in X$, the set
\[ \{ F_t^{(n)}(x) ~|~ t \in K\} \] is compact in $Y$, hence bounded.  By the Banach-Steinhaus theorem, the set $\{F_t^{(n)}: X \to Y\}_{t \in K}$ is equicontinuous.  Thus for any continuous seminorm $q$ on $Y$, there exists a continuous seminorm $p$ on $X$ such that
\[ q(F_t^{(n)}(x)) \leq p(x), \qquad \forall x \in X, \quad \forall t \in K. \]  Consequently,
\[ \sup_{t \in K} q(F_t^{(n)}(x)) \leq p(x), \qquad \forall x \in X.\] The expression on the left, which depends on $K, n,$ and $q$, is one of the defining seminorms of $C^\infty(J,Y)$ applied to $\widetilde{F}(x)$.  The topology of $C^\infty(J,Y)$ is generated by all such seminorms as $K, n,$ and $q$ vary.  This shows that $\widetilde{F}$ is continuous.  
\end{proof}

One can also view a collection $\{F_t: X \to Y\}_{t \in J}$ as a map $F: J \to \Hom(X,Y)$.

\begin{proposition} \label{Proposition-SmoothPathsInHom}
Consider a collection $\{F_t: X \to Y\}_{t \in J}$ of continuous linear maps and the corresponding map $F: J \to \Hom(X,Y)$.
\begin{enumerate}
\item If $\{F_t\}_{t \in J}$ is a smooth family, then $F: J \to \Hom(X,Y)$ is a smooth curve.
\item If the curve $F: J \to \Hom(X,Y)$ is smooth, then $F: J \to \Hom_\sigma(X,Y)$ is smooth (with respect to the topology of pointwise convergence).
\item If $X$ is barreled and the curve $F: J \to \Hom_\sigma(X,Y)$ is smooth, then $\{F_t\}_{t \in J}$ is a smooth family.
\end{enumerate}
\end{proposition}

\begin{proof}
The second statement is trivial and the third is a restatement of Proposition \ref{Proposition-SmoothFamilyForBarreledX}.  For the first statement, we shall prove $F$ is differentiable, and the proof that $F$ is $n$ times differentiable follows by replacing $F$ with $F^{(n-1)}$.  Fix $t \in J$, $\epsilon > 0$, and a continuous seminorm $q$ on $Y$.  By continuity of $\widetilde{F}$, there is a seminorm $p$ on $X$ such that
\[ \sup_{u \in [t-\epsilon, t+\epsilon]} q(F_u''(x)) \leq p(x), \qquad \forall x \in X. \]
From Taylor's formula
\[ F_{t+h}(x) = F_t(x) + F_t'(x)h + \int_t^{t+h} F_u''(x)(t+h-u)du, \] we see
\[ q\left( \frac{F_{t+h}(x) - F_t(x)}{h} - F_t'(x) \right) \leq \sup_{u \in [t,t+h]} q(F_u''(x))\frac{h}{2} \leq p(x)\frac{h}{2} \] for $|h| < \epsilon$.  Given a bounded subset $A \subset X$, consider the seminorm on $\Hom(X,Y)$
\[ q_A(G) = \sup_{x \in A} q(G(x)), \qquad \forall G \in \Hom(X,Y). \]  Then we have shown
\[ q_A\left( \frac{F_{t+h} - F_t}{h} - F_t'\right) \leq C_A\frac{h}{2}, \] where $C_A = \sup_{x \in A}p(x) < \infty$.  Since the right side goes to $0$ as $h \to 0$, this proves $\frac{d}{dt} F_t = F_t'$ in the topology of $\Hom(X,Y)$.
\end{proof}

\begin{corollary} \label{Corollary-BanachSpaceInvertibleSmoothFamily}
Let $X$ be a Banach space and let $\{F_t: X \to X\}_{t \in J}$ be a smooth family of continuous linear maps such that each $F_t$ is bijective.  Then $\{F_t^{-1}\}_{t \in J}$ is a smooth family as well.  Consequently, the map $F: C^\infty(J,X) \to C^\infty(J,X)$ induced by $\{F_t\}_{t \in J}$ is a topological isomorphism of $C^\infty(J)$-modules.
\end{corollary}

\begin{proof}
That each $F_t^{-1}$ is continuous follows from the open mapping theorem.  It is well-known that the inversion map on the set of invertibles of the Banach algebra $\Hom(X,X)$ is differentiable.  If we view $\{F_t\}_{t \in J}$ as a differentiable path in $\Hom(X,X)$, then it follows from the chain rule that the path corresponding to $\{F_t^{-1}\}_{t \in J}$ is differentiable.  From Proposition \ref{Proposition-SmoothPathsInHom}, $\{F_t^{-1}\}_{t \in J}$ is a smooth family.  The induced endomorphism of $C^\infty(J,X)$ is clearly inverse to $F$.
\end{proof}

We've described three different meanings for maps $\{F_t: X \to Y\}_{t \in J}$ to depend smoothly on $t$.  The corresponding inclusions
\[ \Hom_{C^\infty(J)}\left(C^\infty(J,X), C^\infty(J,Y)\right) \to C^\infty\left(J, \Hom(X,Y)\right) \to C^\infty\left(J,\Hom_\sigma(X,Y)\right) \]
are continuous.  They are linear, but not necessarily topological, isomorphisms when $X$ is barreled.

\begin{proposition} \label{Proposition-BanachSpaceSmoothPathsOfHoms}
The canonical map
\[ \Hom_{C^\infty(J)}\left(C^\infty(J,X), C^\infty(J,Y)\right) \to C^\infty\left(J, \Hom(X,Y)\right) \]
is a topological isomorphism if either
\begin{enumerate}
\item $X$ and $Y$ are Banach spaces, or
\item $X$ is a nuclear Fr\'{e}chet space and $Y \in \LCTVS$.
\end{enumerate}
\end{proposition}

\begin{proof}
If $X$ and $Y$ are Banach spaces, then we claim that the domain and codomain are both Fr\'{e}chet spaces.  Then the result follows from the open mapping theorem.  Since $\Hom(X,Y)$ is a Banach space, $C^\infty\left( J, \Hom(X,Y)\right)$ is a Fr\'{e}chet space.  Proposition \ref{Proposition-FreeModuleUMPTopologicalIsomorphism} gives a topological isomorphism
\[ \Hom_{C^\infty(J)}\left(C^\infty(J,X), C^\infty(J,Y)\right) \cong \Hom\left(X, C^\infty(J,Y) \right), \] and the topology of the latter is generated by a countable family of seminorms.

If $X$ is a nuclear Fr\'{e}chet space, then there is a topological isomorphism
\[ X^* \potimes Z \cong \Hom(X, Z) \]
for any $Z \in \LCTVS$, see \cite[Proposition 50.5]{MR0225131}.  Using this and Proposition \ref{Proposition-FreeModuleUMPTopologicalIsomorphism}, we have topological isomorphisms
\begin{align*}
\Hom_{C^\infty(J)} \left( C^\infty(J,X), C^\infty(J,Y) \right) &\cong \Hom\left(X, C^\infty(J,Y) \right)\\
&\cong X^* \potimes C^\infty(J) \potimes Y\\
&\cong C^\infty(J) \potimes \Hom(X,Y)\\
&\cong C^\infty\left(J, \Hom(X,Y)\right).
\end{align*}
\end{proof}


An important case to consider is when $Y = \C$.  Recall that $M^\dual$ denotes the topological $C^\infty(J)$-linear dual of a $C^\infty(J)$-module $M$.

\begin{corollary} \label{Corollary-DualFreeModule}
If $X$ is either a Banach space or a nuclear Fr\'{e}chet space, then
\[ C^\infty(J,X)^\dual \cong C^\infty(J, X^*). \]
\end{corollary}

\section{Connections and parallel translation}

\subsection{Connections}

Since we are only dealing with one-parameter deformations, we shall only treat connections on $C^\infty(J)$-modules where the interval $J$ represents the parameter space.  As there is only one direction to differentiate in, a connection is determined by its covariant derivative.  In what follows, we shall identify the two notions, and will commonly refer to covariant differential operators as connections.
\begin{definition}
A \emph{connection} on a locally convex $C^\infty(J)$-module $M$ is a continuous $\C$-linear map $\nabla: M \to M$ such that
\[ \nabla( f\cdot m) = f' \cdot m + f\cdot \nabla m \]
for all $f \in C^\infty(J)$ and $m \in M$.
\end{definition}
It is immediate from this Leibniz rule that the difference of two connections is a continuous $C^{\infty}(J)$-linear map.  Further, given any connection $\nabla$ and continuous $C^{\infty}(J)$-linear map $F:M \to M$, the operator $\nabla - F$ is also a connection.  So if the space of connections is nonempty, then it is an affine space parametrized by the space $\End_{C^{\infty}(J)}(M)$ of continuous $C^\infty(J)$-linear endomorphisms.  Since the operator $\frac{d}{dt}$ is an example of a connection on a free module $C^\infty(J,X)$, we obtain the following classification.

\begin{proposition} \label{Proposition-FormOfConnection}
If $\nabla$ is a connection on $C^\infty(J,X)$, where $X \in \LCTVS$, then
\[ \nabla = \frac{d}{dt} - F, \]
for some continuous $C^\infty(J)$-linear map $F: C^\infty(J,X) \to C^\infty(J,X)$.
\end{proposition}

An element in the kernel of a connection $\nabla$ will be called a \emph{parallel section} for $\nabla$.  Suppose $M$ and $N$ are two locally convex $C^\infty(J)$-modules with connections $\nabla_M$ and $\nabla_N$ respectively.  We shall say that a continuous $C^\infty(J)$-linear map $F: M \to N$ is \emph{parallel} if $F \circ \nabla_M = \nabla_N \circ F$.  
A parallel map sends parallel sections to parallel sections.

\begin{proposition} \label{Proposition-TensorAndDualConnections}
Given locally convex $C^\infty(J)$-modules $M$ and $N$ with connections $\nabla_M$ and $\nabla_N$,
\begin{enumerate}[(i)]
\item \label{Part-TensorConnection}
 the operator $\nabla_M \otimes 1 + 1 \otimes \nabla_N$ is a connection on $M \potimes_{C^\infty(J)} N$.
\item \label{Part-DualConnection}
the operator $\nabla_M^\dual$ on $M^\dual = \Hom_{C^\infty(J)}(M, C^\infty(J))$ given by
\[ (\nabla_M^\dual\phi)(m) = \frac{d}{dt} \phi(m) - \phi(\nabla_M m) \]
is a connection.
\end{enumerate}
\end{proposition}

The definition of $\nabla_M^\dual$ ensures that the canonical pairing
\[ \langle \cdot , \cdot \rangle : M^\dual \otimes_{C^\infty(J)} M \to C^\infty(J) \] is a parallel map, where we consider the ground ring $C^\infty(J)$ with the connection $\frac{d}{dt}$.  That is,
\[ \frac{d}{dt} \langle \phi, m \rangle = \langle \nabla_{M^*}\phi, m \rangle + \langle \phi, \nabla_M m \rangle. \]

\subsection{Parallel translation in free modules}

Let $X \in \LCTVS$ and let $M = C^\infty(J,X)$ be the corresponding free module.
\begin{definition}
A connection $\nabla$ on $M$ is \emph{integrable} if there is a parallel isomorphism
\[ F: (M, \nabla) \to \left(C^\infty(J,X), \frac{d}{dt}\right) \] of locally convex $C^\infty(J)$-modules.
\end{definition}
We shall express this condition in terms of parallel translation.  We will think of $M$ as sections of the trivial bundle whose fiber over $t \in J$ is $M_t \cong X$.  Parallel translation relies on the existence and uniqueness of a solution $m \in M$ to the initial value problem
\[ \nabla m = 0, \qquad m(s) = x \] for any given $s \in J$ and $x \in M_s$.  In this case, the parallel translation operator
\[ P^\nabla_{s,t}: M_s \to M_t \] is the linear map defined by $P^\nabla_{s,t}(x) = m(t)$, where $m$ is the unique solution to the above initial value problem.  Evidently, $P^\nabla_{s,t}$ is a linear isomorphism with inverse $P^\nabla_{t,s}$.

\begin{theorem} \label{Theorem-IntegrableConnections}
A connection $\nabla$ on $M = C^\infty(J,X)$ is integrable if and only if the following two conditions hold:
\begin{enumerate}[(i)]
\item For every $s \in J$ and $x \in M_s$, there is a unique $m \in M$ such that
\[ \nabla m = 0, \qquad m(s) = x. \]
\item The linear map $P^\nabla: X \to C^\infty(J \times J, X)$ given by
\[ P^\nabla(x)(s,t) = P^\nabla_{s,t}(x) \] is well-defined and continuous.
\end{enumerate}
\end{theorem}

\begin{proof}
Notice that the connection $\frac{d}{dt}$ on $C^\infty(J,X)$ satisfies both conditions.  Moreover, both conditions are preserved by parallel isomorphism.  So an integrable connection $\nabla$ satisfies (i) and (ii).

Conversely, suppose $\nabla$ satisfies (i) and (ii) and fix a value $s \in J$.  By condition (ii), the linear maps
\[ \widetilde{F}: X \to C^\infty(J,M_s), \qquad \widetilde{F}(x)(t) = P^\nabla_{t,s}(x) \]
\[ \widetilde{G}: M_s \to M, \qquad \widetilde{G}(x)(t) = P^\nabla_{s,t}(x) \]
are continuous, and induce mutually inverse $C^\infty(J)$-linear isomorphisms
\[ F: M \to C^\infty(J, M_s), \qquad G: C^\infty(J,M_s) \to M \]
by the universal property of free modules.  We'll show that
\[ G: \left(C^\infty(J,M_s), \frac{d}{dt}\right) \to \left(M, \nabla\right) \]
is parallel.  By $C^\infty(J)$-linearity, the Leibniz rule, and continuity, it suffices to check
\[ G \circ \frac{d}{dt} = \nabla \circ G \] for elements of the form $1 \otimes x \in C^\infty(J) \potimes M_s$.  But this follows immediately by definition of parallel translation.  That $F = G^{-1}$ is parallel follows automatically.
\end{proof}

The second condition in the theorem can be weakened if $X$ is barreled.

\begin{theorem}
If $X \in \LCTVS$ is barreled, then a connection $\nabla$ on $M = C^\infty(J,X)$ is integrable if and only if the following two conditions hold:
\begin{enumerate}[(i)]
\item For every $s \in J$ and $x \in M_s$, there is a unique $m \in M$ such that
\[ \nabla m = 0, \qquad m(s) = x. \]
\item Each $P^\nabla_{s,t}: M_s \to M_t$ is continuous, and for each fixed $x \in X$, the map $(s,t) \mapsto P^\nabla_{s,t}(x)$ is smooth (i.e. all mixed partial derivatives exist).
\end{enumerate}
\end{theorem}

\begin{proof}
Mimic the proof of Proposition~\ref{Proposition-SmoothFamilyForBarreledX} to show that
\[ P^\nabla: X \to C^\infty(J\times J, X), \qquad P^\nabla(x)(s,t) = P^\nabla_{s,t}(x) \] is continuous.
\end{proof}

Essentially, an integrable connection $\nabla$ is one for which we can parallel translate, and moreover the parallel translation operators $P^\nabla_{s,t}$ are isomorphisms of topological vector spaces that depend smoothly on both parameters $s$ and $t$.

Now suppose $X, Y \in \LCTVS$ and $M = C^\infty(J,X)$ and $N = C^\infty(J,Y)$.  Suppose $F: M \to N$ is a continuous $C^\infty(J)$-linear map and $\{F_t: M_t \to N_t\}_{t \in J}$ is the corresponding smooth family of continuous linear maps. 

\begin{proposition} \label{Proposition-ParallelFamilyCommutesWithParallelTransport}
In the above situation, if $F: M \to N$ is parallel with respect to integrable connections on $M$ and $N$, then the diagram
\[ \xymatrix{
M_s \ar[r]^-{F_s} \ar[d]_-{P^{\nabla_M}_{s,t}} & N_s \ar[d]^-{P^{\nabla_N}_{s,t}}\\
M_t \ar[r]^-{F_t} & N_t\\
} \]
commutes for all $t,s \in J$.
\end{proposition}

\begin{proof}
Given $x \in M_s$, let $m \in M$ be the unique $\nabla_M$-parallel section through $x$.  Then $F(m)$ is the unique $\nabla_N$-parallel section through $F(m)(s) = F_s(x)$.  Consequently,
\[ P^{\nabla_N}_{s,t}(F_s(x)) = F(m)(t) = F_t(m(t)) = F_t(P^{\nabla_M}_{s,t}(x)). \]
\end{proof}

If $N$ has the trivial connection $\frac{d}{dt}$, we obtain the following.

\begin{corollary} \label{Corollary-ParallelTranslationFormula}
Given any integrable connection $\nabla$ on $M$ and parallel isomorphism
\[ F: (M, \nabla) \to \left(C^\infty(J,X), \frac{d}{dt}\right), \]
then $P^\nabla_{s,t} = F_t^{-1} \circ F_s: M_s \to M_t$.
\end{corollary}

\begin{proposition} \label{Proposition-IntegrableTensorAndDual}
\begin{enumerate}
\item If $\nabla_M$ and $\nabla_N$ are integrable connections on $M$ and $N$, then $\nabla_{\potimes} := \nabla_M \otimes 1 + 1 \otimes \nabla_n$ is integrable on $M \potimes_{C^\infty(J)} N$, and
\[ P^{\nabla_{\potimes}}_{s,t} = P^{\nabla_M}_{s,t} \otimes P^{\nabla_N}_{s,t}: M_s \potimes N_s \to M_t \potimes N_t. \]
\item If $X$ is either a Banach space or a nuclear Fr\'{e}chet space, and $M = C^\infty(J,X)$ has an integrable connection $\nabla_M$, then the dual connection $\nabla_M^\dual$ is integrable on $M^\dual = C^\infty(J,X^*)$ (see Corollary \ref{Corollary-DualFreeModule} for this identification), and
\[ P^{\nabla_M^\dual}_{s,t} = \left(P^\nabla_{t,s}\right)^*: M_s^* \to M_t^*. \]
\end{enumerate}
\end{proposition}

\begin{proof}
Given parallel isomorphisms
\[ F: (M, \nabla_M) \to \left(C^\infty(J,X), \frac{d}{dt}\right), \qquad G: (N, \nabla_N) \to \left(C^\infty(J,Y), \frac{d}{dt}\right), \] we obtain a parallel isomorphism
\begin{align*}F \otimes G: \left( M \potimes_{C^\infty(J)} N, \nabla_{\potimes} \right) &\to \left( C^\infty(J,X) \potimes_{C^\infty(J)} C^\infty(J,Y), \frac{d}{dt} \otimes 1 + 1 \otimes \frac{d}{dt} \right)\\
&\cong \left( C^\infty(J, X \potimes Y), \frac{d}{dt} \right),
\end{align*} which shows $\nabla_{\potimes}$ is integrable.  In a similar way, we obtain a parallel isomorphism
\[ (F^{-1})^\dual: (M^\dual, \nabla_M^\dual) \to \left( C^\infty(J, X^*), \frac{d}{dt} \right) \]
because the dual connection of $\frac{d}{dt}$ on $C^\infty(J,X)$ identifies with $\frac{d}{dt}$ on $C^\infty(J,X^*)$ under the isomorphism of Corollary \ref{Corollary-DualFreeModule}.  The parallel translation formulas follow from Corollary \ref{Corollary-ParallelTranslationFormula}.
\end{proof}

Let us consider the problem of parallel translation for a connection $\nabla$ on $M = C^\infty(J,X)$.  Recall by Proposition~\ref{Proposition-FormOfConnection} that $\nabla = \frac{d}{dt} - F$ for some continuous $C^\infty(J)$-linear map $F: C^\infty(J,X) \to C^\infty(J,X)$.  Let $\{F_t: X \to X\}_{t \in J}$ be the corresponding smooth family of continuous linear maps.  To parallel translate, we must solve the first order linear ODE
\[ x'(t) = F_t(x(t)), \qquad x(s) = x_0 \]
given $s \in J$ and $x_0 \in X$.  By the fundamental theorem of calculus (which is valid for functions with values in $X \in \LCTVS$), any solution satisfies
\[ x(t) = x(s) + \int_s^tx'(u)du = x_0 + \int_s^t F_u(x(u))du. \]  Applying the fundamental theorem inductively, we obtain
\begin{align*}
x(t) &= x_0 + \sum_{n=1}^N \int_{s}^t \int_{s}^{u_1} \ldots \int_{s}^{u_{n-1}} (F_{u_1} \circ \ldots \circ F_{u_n})(x_0) du_n \ldots du_1\\ &\qquad \qquad+ \int_{s}^t \int_{s}^{u_1} \ldots \int_{s}^{u_{N}}(F_{u_1} \circ \ldots \circ F_{u_{N+1}})(x(u_{N+1})) du_{N+1}du_N \ldots du_1. \end{align*} for any $N$.  If the last term can be shown to converge to $0$ in $C^\infty(J,X)$ as $N \to \infty$, then any solution $x(t)$ has the form
\[ x(t) = x_0 + \sum_{n=1}^\infty \int_{s}^t \int_{s}^{u_1} \ldots \int_{s}^{u_{n-1}} (F_{u_1} \circ \ldots \circ F_{u_n})(x_0) du_n \ldots du_1. \]  This gives uniqueness of solutions.  If this series can be shown to converge, we obtain existence of solutions.  It is straightforward to show both of these in the case where $X$ is a Banach space.  The fundamental theorem of calculus ensures that the solution depends smoothly on both $t$ and $s$.  These are well-known results from the theory of first order linear ODE's on a Banach space, which we restate in our language.

\begin{theorem}
If $X$ is a Banach space, then every connection on $C^\infty(J,X)$ is integrable.
\end{theorem}

Notice that if $\nabla$ has constant coefficients, i.e. $F_t$ doesn't depend on $t$, then the solution takes the well-known form
\[ x(t) = \sum_{n=0}^\infty \frac{(t-s)^n}{n!}F^n(x_0) = \exp((t-s)F)(x_0). \]

Once we start considering other classes of locally convex vector spaces, e.g. Fr\'{e}chet spaces, the existence and uniqueness theorem for solutions to linear ODE's is false.  One cannot guarantee that the above series defining the solution will converge.

Another situation in which we can get control of this series is when $F$ has nilpotence properties.  We'll call $F$ \emph{nilpotent (with respect to $\frac{d}{dt}$)} if there is an integer $n$ such that
\[ F_{u_1} \circ \ldots \circ F_{u_n} = 0, \qquad \forall u_1, \ldots , u_n \in J. \]  In this case, the above series becomes a finite sum, and we see that the connection $\nabla = \frac{d}{dt} - F$ is integrable.

Let's generalize the above discussion to perturbations $\nabla - F$ of an integrable connection $\nabla$.

\begin{theorem}[Fundamental theorem of calculus]
If $\nabla$ is an integrable connection on $C^\infty(J,X)$, then
\[ x(t) = P_{s,t}^\nabla(x(s)) + \int_s^t P_{u,t}^\nabla\left( (\nabla x)(u) \right)du \] for any $x \in C^\infty(J,X)$ and $s,t \in J$.
\end{theorem}

\begin{proof}
Fix $s \in J$ and view everything as a function of $t$ (so that $\nabla$ differentiates with respect to $t$).  Using the fact that $(\nabla \circ P_{u,t})(y(u)) = 0$ for any $u$ or $y$, we see that applying $\nabla$ to the right hand side gives $(\nabla x)(t)$.  Thus the two sides differ by a $\nabla$-parallel section, which must be $0$ because the two sides are equal when $t = s$.
\end{proof}

By repeatedly applying this fundamental theorem of calculus, we see that solutions to
\[ (\nabla - F)x = 0, \qquad x(s) = x_0 \]
take the form
\begin{align*} x(t) &= P_{s,t}^\nabla(x_0) + \sum_{n=1}^\infty \int_s^t \int_s^{u_1} \ldots \int_s^{u_{n-1}}\\ &\qquad(P_{u_1,t}^\nabla \circ F_{u_1} \circ P_{u_2,u_1}^\nabla \circ F_{u_2} \circ \ldots \circ P_{u_n,u_{n-1}}^\nabla \circ F_{u_n} \circ P_{s,u_{n}}^\nabla)(x_0)du_n \ldots du_1,
\end{align*} provided the series converges.
So we shall say that $F$ is \emph{nilpotent (with respect to $\nabla$)} if there is an integer $n$ such that
\[ F_{u_1} \circ P_{u_2,u_1}^\nabla \circ F_{u_2} \circ \ldots \circ P_{u_n,u_{n-1}}^\nabla \circ F_{u_n} = 0, \qquad \forall u_1, \ldots u_n \in J. \]  Thus a nilpotent perturbation $\nabla - F$ of an integrable connection is integrable.  We are interested in a special case of this.  We record it here, though it shall be used in a subsequent paper.

\begin{proposition}\label{Proposition-NilpotentPerturbation}
Suppose $\nabla$ is an integrable connection on $C^\infty(J,X)$, and $F$ is a $C^\infty(J)$-linear endomorphism of $C^\infty(J,X)$ such that $[\nabla, F] = 0$ and $F^N = 0$ for some integer $N$.  Then $\nabla - F$ is integrable and
\[ P_{s,t}^{\nabla - F} = \sum_{n=0}^{N-1} \frac{(t-s)^n}{n!}F_t^n \circ P_{s,t}^\nabla. \]
\end{proposition}

\begin{proof}
By assumption, $F$ is parallel with respect to $\nabla$.  Using Proposition~\ref{Proposition-ParallelFamilyCommutesWithParallelTransport}, we have
\[ P_{u_1,t}^\nabla \circ F_{u_1} \circ P_{u_2,u_1}^\nabla \circ F_{u_2} \circ \ldots \circ P_{u_n,u_{n-1}}^\nabla \circ F_{u_n} \circ P_{s,u_{n}}^\nabla = F_t^n \circ P_{s,t}^\nabla. \]
It follows that $F$ is nilpotent with respect to $\nabla$, so $\nabla - F$ is integrable.  From the explicit series solution, we see
\begin{align*} P_{s,t}^{\nabla-F}(x_0) &= P_{s,t}^\nabla(x_0) + \sum_{n=1}^{N-1} \int_s^t \int_s^{u_1} \ldots \int_s^{u_{n-1}} F_t^n(P_{s,t}^\nabla(x_0))du_n\ldots du_1\\
&= \sum_{n=0}^{N-1} \frac{(t-s)^n}{n!}F_t^n(P_{s,t}^\nabla(x_0)).
\end{align*}
\end{proof}
In other words, if $x$ is the $\nabla$-parallel section through $x_0$ over $s \in J$, then $\exp((t-s)F)(x)$ is the $(\nabla - F)$-parallel section through $x_0$.

\section{Smooth deformations}

\subsection{Deformations of algebras}

Let $X \in \LCTVS$ and let $J$ denote an open interval of real numbers.

\begin{definition}
A \emph{smooth one-parameter deformation of algebras} is a smooth family of continuous linear maps $\{ m_t: X \potimes X \to X \}_{t \in J}$ for which each $m_t$ is associative.
\end{definition}
So for each $t \in J$, we have a locally convex algebra $A_t := (X, m_t)$ whose underlying space is $X$.  
Consider the continuous $C^\infty(J)$-linear map
\[ m: C^\infty(J, X \potimes X) \to C^\infty(J,X) \] associated to the smooth family $\{m_t\}_{t \in J}$.  Letting $A_J = C^\infty(J,X)$, then $m$ can be viewed as a map
\[ m: A_J \potimes_{C^\infty(J)} A_J \to A_J \]
using Proposition~\ref{Proposition-TensorProductOfFreeModules}.  Associativity of $m$ follows from associativity of the family $\{m_t\}$.  Thus $A_J$ is a locally convex $C^\infty(J)$-algebra, which we shall refer to as the \emph{algebra of sections} of the deformation $\{A_t\}_{t \in J}$.  Explicitly, the multiplication in $A_J$ is given by
\[ (a_1 a_2)(t) = m_t(a_1(t), a_2(t)) \]  for all $a_1, a_2 \in A_J$.  Note that the evaluation maps $\epsilon_t: A_J \to A_t$ are algebra homomorphisms.

\begin{proposition}
Associating to a deformation its algebra of sections gives a one-to-one correspondence between smooth one-parameter deformations over $J$ with underlying space $X$ and locally convex $C^\infty(J)$-algebra structures on $C^\infty(J,X)$.
\end{proposition}

If $X$ is Fr\'{e}chet, then our definition of a smooth deformation is equivalent to a smooth path in $\Hom(X \potimes X, X)$ whose image lies in the set of associative products, by Proposition \ref{Proposition-SmoothPathsInHom}.  The following is a useful criterion for checking that the deformation is smooth in this case.

\begin{proposition} \label{Proposition-DeformationsOnFrechetSpace}
If $X$ is a Fr\'{e}chet space, then a set of continuous associative multiplications $\{m_t: X \potimes X \to X\}_{t \in J}$ is a smooth one-parameter deformation if and only if the map
\[ t \mapsto m_t(x_1, x_2) \] is smooth for each fixed $x_1, x_2 \in X$.
\end{proposition}

\begin{proof}
If $\{m_t\}_{t \in J}$ is a smooth one-parameter deformation, then it is immediate that $t \mapsto m_t(x_1, x_2)$ is smooth for all $x_1, x_2 \in X$.

Conversely, if $t \mapsto m_t(x_1, x_2)$ is smooth for each fixed $x_1, x_2 \in X$, then the map
\[ m: X \times X \to C^\infty(J,X) \] given by
\[ m(x_1, x_2)(t) = m_t(x_1, x_2) \]
is separately continuous by Proposition~\ref{Proposition-SmoothFamilyForBarreledX}.  Since $X$ is Fr\'{e}chet, it follows that $m$ is jointly continuous and so induces a continuous linear map
\[ m: X \potimes X \to C^\infty(J,X). \]  This shows that $\{m_t\}_{t \in J}$ is a smooth family of continuous linear maps.
\end{proof}


\begin{definition}
A \emph{morphism} between the deformations $\{A_t\}_{t \in J}$ and $\{B_t\}_{t\in J}$ is a continuous $C^\infty(J)$-linear algebra homomorphism $F: A_J \to B_J$.
\end{definition}

Thus a morphism is equivalent to a family $\{F_t: A_t \to B_t\}_{t \in J}$ of continuous algebra homomorphisms which vary smoothly in the sense of Definition~\ref{Definition-SmoothFamilyOfLinearMaps}.  When $X$ is Fr\'{e}chet, the smoothness can be checked using Proposition~\ref{Proposition-SmoothFamilyForBarreledX}.

A deformation is called \emph{constant} if the products $\{m_t\}$ do not depend on $t$.  A deformation is called \emph{trivial} if it is isomorphic to a constant deformation.  Thus $\{A_t\}_{t \in J}$ is trivial if and only if there is a locally convex algebra $B$ such that $A_J \cong C^\infty(J,B)$ as algebras.  We can characterize triviality of a smooth deformation of algebras in terms of connections.

\begin{proposition} \label{Proposition-TrivialDeformationOfAlgebras}
The deformation $\{A_t\}_{t \in J}$ is trivial if and only if $A_J$ admits an integrable connection that is a derivation with respect to the algebra structure.  In this case, the parallel translation maps $P_{s,t}^\nabla: A_s \to A_t$ are isomorphisms of locally convex algebras.
\end{proposition}

\begin{proof}
Notice that $\frac{d}{dt}$ is an integrable connection and a derivation on a constant deformation.  If $\{A_t\}_{t \in J}$ is trivial and $F: A_J \to B_J$ is a $C^\infty(J)$-linear algebra isomorphism with the algebra of sections of a constant deformation, then $\nabla = F^{-1}\frac{d}{dt}F$ is a connection and a derivation on $A_J$, and $\nabla$ is integrable because
\[ F: (A_J, \nabla) \to \left(B_J, \frac{d}{dt}\right) \] is a parallel isomorphism.

Conversely, suppose $A_J$ has an integrable connection $\nabla$ that is a derivation.  That $\nabla$ is a derivation is equivalent to the multiplication map
\[ m: ( A_J \potimes_{C^\infty(J)} A_J, \nabla \otimes 1 + 1 \otimes \nabla) \to (A_J, \nabla) \] being a parallel map.  Combining Proposition \ref{Proposition-ParallelFamilyCommutesWithParallelTransport} with Corollary \ref{Corollary-ParallelTranslationFormula}, we see that each $P^\nabla_{s,t}: A_s \to A_t$ is an algebra isomorphism.
Fixing an $s \in J$, it follows that 
\[ \{P^\nabla_{s,t}: A_s \to A_t\}_{t \in J} \]
is an isomorphism between the constant deformation with fiber $A_s$ and $\{A_t\}_{t \in J}$.
\end{proof}

%

Thus it is important to determine if a deformation has a connection that is a derivation.  In analogy with the work of Gerstenhaber on formal deformations \cite{MR0171807}, the obstruction to this is cohomological.

Given any connection $\nabla$ on the algebra of sections $A_J$, define the bilinear map $E$ by
\[ \nabla(a_1 a_2) = \nabla(a_1)a_2 + a_1\nabla(a_2) - E(a_1, a_2). \]
So $E$ is the defect of $\nabla$ from being a derivation, and in fact $E = \delta \nabla$, where $\delta$ is the Hochschild coboundary.  It follows that $\delta E = 0$.  Using the Leibniz rule for $\nabla$, one can check that $E$ is a $C^\infty(J)$-bilinear map.  So $E$ defines a cohomology class $[E] \in H^2_{C^\infty(J)}(A_J,A_J)$.  Notice that $\nabla$ is only $\C$-linear and not $C^\infty(J)$-linear.  Thus, we may have $[E] \neq 0$.

\begin{proposition} \label{Proposition-ClassOfE}
The cohomology class $[E] \in H^2_{C^\infty(J)}(A_J,A_J)$ is independent of the choice of connection.  Moreover, $[E] = 0$ if and only if $A_J$ possesses a connection that is a derivation.
\end{proposition}

\begin{proof}
Let $\nabla$ and $\nabla'$ be two connections with corresponding cocycles $E$ and $E'$.  Since $\nabla' = \nabla - F$ for some $F \in C^1_{C^\infty(J)}(A_J,A_J)$, we have
\[ E' = \delta(\nabla - F) = E - \delta F, \] which shows that $[E] = [E'].$

If $\nabla$ is a connection that is a derivation, then $E = \delta \nabla = 0.$  Conversely, if $\nabla$ is any connection on $A_J$ and $[E] = 0$, then $E = \delta F$ for some $F \in C^1_{C^\infty(J)}(A_J,A_J)$.  Hence $\delta (\nabla - F) = 0$, which shows that $\nabla - F$ is a connection that is a derivation.
\end{proof}

From this, we see that the cohomology class $[E]$ provides an obstruction to the triviality of a deformation.  Even if this obstruction vanishes, there is still an analytic obstruction in that the corresponding connection may not be integrable.  These two issues are common to the smooth deformation theory of other types of structures as well, e.g. cochain complexes (see below) or $A_\infty$-algebras \cite{Yashinski-Thesis}.

\subsection{Deformations of cochain complexes}

By a \emph{smooth one-parameter deformation of cochain complexes}, we mean a collection $\{X^n\}_{n \in \Z}$ of spaces in $\LCTVS$ together with a smooth family of continuous linear maps $\{ d^n_t: X^n \to X^{n+1}\}_{t \in J}$ for each $n$ such that $d^{n+1}_t \circ d^n_t = 0$ for all $t \in J$.  (By turning the arrows around, we could just as well talk about deformations of chain complexes.)  For each $t \in J$, we have a locally convex cochain complex
\[ (\mathcal{C}^\bullet_t, d_t) := \left(
\xymatrix{
\ldots \ar[r]^-{d_t} & X^{n-1} \ar[r]^-{d_t} & X^n \ar[r]^-{d_t} & X^{n+1} \ar[r]^-{d_t} &\ldots
}
\right) \]
built on the same underlying family of spaces.
Let $\mathcal{C}_J^n = C^\infty(J,X^n)$ and let $d: \mathcal{C}_J^n \to \mathcal{C}_J^{n-1}$ be the continuous $C^\infty(J)$-linear map associated to the smooth family $\{d_t\}_{t \in J}$.  We obtain a chain complex
\[ (\mathcal{C}_J^\bullet, d) := \left(
\xymatrix{
\ldots \ar[r]^-d & \mathcal{C}_J^{n-1} \ar[r]^-d & \mathcal{C}_J^n \ar[r]^-d & \mathcal{C}_J^{n+1} \ar[r]^-d &\ldots
}
\right) \]
of locally convex $C^\infty(J)$-modules.  We'll call $\mathcal{C}_J^\bullet$ the \emph{complex of sections} of the deformation.  The cohomology $H^\bullet(\mathcal{C}_J)$ is a $C^\infty(J)$-module, and the evaluation chain maps $\epsilon_t: \mathcal{C}_J^\bullet \to \mathcal{C}_t^\bullet$ induce maps on cohomology
\[ (\epsilon_t)_*: H^\bullet(\mathcal{C}_J) \to H^\bullet(\mathcal{C}_t). \]

By a morphism of two deformations, we mean a continuous $C^\infty(J)$-linear (degree $0$) chain map between their respective complexes of sections.  We'll call a deformation \emph{trivial} if it is isomorphic to a constant deformation.

\begin{proposition}
Suppose $(\mathcal{C}^\bullet, d)$ is a cochain complex of Fr\'{e}chet spaces such that the cohomology $H^\bullet(\mathcal{C})$ is Hausdorff.  Let $\mathcal{C}_J^\bullet = C^\infty(J, \mathcal{C}^\bullet)$ be the complex of sections of the constant deformation with fiber $\mathcal{C}^\bullet$.  Then
\[ H^\bullet(\mathcal{C}_J) \cong C^\infty\left(J, H^\bullet(\mathcal{C})\right) \] as locally convex $C^\infty(J)$-modules.
\end{proposition}

\begin{proof}
Notice that requiring $H^n(\mathcal{C}) = Z^n(\mathcal{C}) / B^n(\mathcal{C})$ to be Hausdorff is equivalent to requiring the space of coboundaries $B^n(\mathcal{C})$ to be closed.  In this case, both $B^n(\mathcal{C})$ and $H^n(\mathcal{C})$ are Fr\'{e}chet spaces for all $n$.

Notice that $Z^n(\mathcal{C}_J) = C^\infty(J, Z^n(\mathcal{C}))$, but a priori we only have $B^n(\mathcal{C}_J) \subseteq C^\infty(J, B^n(\mathcal{C}))$.  However, since $d: \mathcal{C}^n \to B^{n+1}(\mathcal{C})$ is a surjection of Fr\'{e}chet spaces, it follows from \cite[Proposition 43.9]{MR0225131} that
\[ 1 \otimes d: C^\infty(J) \potimes \mathcal{C}^n \to C^\infty(J) \potimes B^{n+1}(\mathcal{C}) \] is surjective as well.  That is, $B^n(\mathcal{C}_J) = C^\infty(J, B^n(\mathcal{C}))$ for all $n$.  Thus,
\[ H^n(\mathcal{C}_J) = Z^n(\mathcal{C}_J)/B^n(\mathcal{C}_J) = C^\infty(J, Z^n(\mathcal{C})) / C^\infty(J, B^n(\mathcal{C})) \cong C^\infty(J, H^n(\mathcal{C})), \] where the last isomorphism is from Proposition \ref{Proposition-QuotientOfFreeModules}.

\end{proof}

%

\begin{example} \label{Example-CyclicComplexOfDeformation}
If $\{A_t\}_{t \in J}$ is a smooth deformation of algebras, then $\{(C_{\per}(A_t), b_t + B)\}_{t \in J}$ is a smooth deformation of chain complexes.  Notice that the Hochschild boundary $b_t$ depends on the multiplication of $A_t$, whereas the operator $B$ does not.  Since the completed projective tensor product commutes with direct products \cite[Theorem 15.4.1]{MR632257}, the complex of sections of $\{C_{\per}(A_t)\}_{t \in J}$ is naturally identified with the periodic cyclic complex $C_{\per}^{C^\infty(J)}(A_J)$.  One can also consider the complexes associated to the various other homology/cohomology theories discussed above.
\end{example}




As in the algebra case, we can characterize triviality of a deformation of chain complexes in terms of connections.  The proofs here are analogous those in the algebra case.

\begin{proposition} \label{Proposition-BundleOfComplexesParallelTransport}
A smooth deformation of cochain complexes $\{\mathcal{C}^\bullet_t\}_{t \in J}$ is trivial if and only if the complex of sections $\mathcal{C}_J^\bullet$ admits an integrable connection that is a chain map.  For such a connection $\nabla$, the parallel translation map $P^\nabla_{s,t}: \mathcal{C}^\bullet_s \to \mathcal{C}^\bullet_t$ is an isomorphism of locally convex cochain complexes for all $s,t \in J$.  In particular, the parallel translation maps induce isomorphisms
\[ (P^\nabla_{s,t})_*: H^\bullet(\mathcal{C}_s) \to H^\bullet(\mathcal{C}_t). \]
\end{proposition}

The obstruction to the existence of such a connection is again cohomological.  Let $\nabla$ be any connection on $\mathcal{C}_J^\bullet$, and consider the map
\[ G = [d, \nabla]: \mathcal{C}_J^\bullet \to \mathcal{C}_J^{\bullet + 1}, \]
which is the defect of $d$ from being a $\nabla$-parallel map (equivalently, the defect of $\nabla$ from being a chain map).  It follows that $G$ is $C^\infty(J)$-linear and $[d, G] = 0$, so $G$ is a cocycle in the endomorphism complex $\End_{C^\infty(J)}(\mathcal{C}_J)$.

\begin{proposition} \label{Proposition-ClassOfG}
The cohomology class $[G] \in H^1(\End_{C^\infty(J)}(\mathcal{C}_J))$ is independent of the choice of connection $\nabla$.  Moreover, $[G] = 0$ if and only if $\mathcal{C}_J^\bullet$ admits a connection that is a chain map.
\end{proposition}


Suppose $\mathcal{C}_J^\bullet$ is equipped with a connection $\nabla$ that is a chain map.  Our main goal is to identify the cohomology groups $H^\bullet(\mathcal{C}_s) \cong H^\bullet(\mathcal{C}_t)$ of different fibers via parallel translation.  By Proposition \ref{Proposition-BundleOfComplexesParallelTransport}, this happens when $\nabla$ is integrable.  In this case, the cochain complexes themselves are fiberwise isomorphic, and this may be too strong of a condition to be useful in practice.

As $\nabla$ is a chain map, it induces a connection $\nabla_*$ on the $C^\infty(J)$-module $H^\bullet(\mathcal{C}_J)$.  The homology module may not be free, complete, or even Hausdorff, so we should be careful about what we mean by integrability of $\nabla_*$.  Nonetheless, it makes sense to inquire about the existence and uniqueness of a solution $[c] \in H^\bullet(\mathcal{C}_J)$ to the cohomological differential equation
\[ \nabla_* [c] = 0, \qquad [c(s)] = [c_0] \]
with initial value $[c_0] \in H^\bullet(\mathcal{C}_s)$.  Having this is enough to construct parallel translation operators
\[ P^{\nabla_*}_{s,t}: H^\bullet(\mathcal{C}_s) \to H^\bullet(\mathcal{C}_t), \]
which are linear isomorphisms.  Additionally, if the map $[c_0] \mapsto [c]$ is continuous, then $P^{\nabla_*}_{s,t}$ is continuous, hence an isomorphism of topological vector spaces.

\section{Some rigidity results}

The results in this section are largely not original.  Most are stated, with slight variations, in \cite{Crainic}, where they are proved using the homological perturbation lemma.  We give different proofs using our methods in the setting of smooth deformations.

We'll call a locally convex algebra $A$ \emph{(smoothly) rigid} if every smooth deformation $\{A_t\}_{t \in J}$ with $A_0 = A$ is trivial on some interval $J' \subset J$ containing $0$.  Similarly, we can define rigidity of a cochain complex.  Our main tool for proving rigidity results in the setting of Banach spaces is the following lemma.

\begin{lemma} \label{Lemma-DeformationOfContractibleHasZeroCohomology}
Let $\{(\mathcal{C}_t^\bullet, d_t)\}_{t \in J}$ be a smooth deformation of cochain complexes of Banach spaces, and suppose the complex $\mathcal{C}_0^\bullet$ has a continuous linear contracting homotopy in degree $n$
\[ \xymatrix{
\mathcal{C}_0^{n-1} \ar@<.5ex>[r]^-{d_0} & \mathcal{C}_0^n \ar@<.5ex>[l]^-h \ar@<.5ex>[r]^-{d_0}& \mathcal{C}_0^{n+1}, \ar@<.5ex>[l]^-h
} \qquad d_0 h + h d_0 = 1.\]
Then there is a subinterval $J' = (-\epsilon, \epsilon)$ such that \[ Z^n(\mathcal{C}_{J'}) \cong C^\infty(J', Z^n(\mathcal{C}_0)), \qquad B^n(\mathcal{C}_{J'}) \cong C^\infty(J', B^n(\mathcal{C}_0)), \qquad H^n(\mathcal{C}_{J'}) = 0.\]
\end{lemma}

\begin{proof}
By assumption, there are split short exact sequences
\[ \xymatrix{
0 \ar[r] & Z^{n-1}(\mathcal{C}_0) \ar[r] &\mathcal{C}_0^{n-1} \ar@<.5ex>[r]^-{d_0} & B^n(\mathcal{C}_0) \ar@<.5ex>[l]^-h \ar[r] &0,
} \]
\[ \xymatrix{
0 \ar[r] & Z^n(\mathcal{C}_0) \ar[r] &\mathcal{C}_0^n \ar@<.5ex>[r]^-{d_0} & B^{n+1}(\mathcal{C}_0) \ar@<.5ex>[l]^-h \ar[r] &0,
} \]
and $B^n(\mathcal{C}_0) = Z^n(\mathcal{C}_0)$.  So the cocycles are complemented in the space of cochains, that is, there are closed subspaces $W^{n-1} = h(B^n(\mathcal{C}_0))$ and $W^n = h(B^{n+1}(\mathcal{C}_0))$ for which
\[ \mathcal{C}_0^{n-1} = Z^{n-1}(\mathcal{C}_0) \oplus W^{n-1}, \qquad \mathcal{C}_0^n = Z^n(\mathcal{C}_0) \oplus W^n. \]
Let $\pi: \mathcal{C}_0^n \to Z^n(\mathcal{C}_0) = B^n(\mathcal{C}_0)$ be the projection.  Then $\pi$ induces a continuous $C^\infty(J)$-linear map $\pi: \mathcal{C}_J^n \to C^\infty(J, Z^n(\mathcal{C}_0))$, which restricts to the maps
\[ \pi: Z^n(\mathcal{C}_J) \to C^\infty(J, Z^n(\mathcal{C}_0)), \qquad \pi: B^n(\mathcal{C}_J) \to C^\infty(J, B^n(\mathcal{C}_0)). \]
We claim these are topological isomorphisms for a small enough interval $J'$ containing $0$.  We'll prove this by showing $\pi$ is injective on cocycles and surjective on coboundaries.  The results then follow from the commutative diagram
\[ \xymatrix{
B^n(\mathcal{C}_{J'}) \ar@{^{(}->}[r] \ar[d]^-\pi &Z^n(\mathcal{C}_{J'}) \ar[d]^-\pi \\
C^\infty(J', B^n(\mathcal{C}_0)) \ar@{=}[r] & C^\infty(J', Z^n(\mathcal{C}_0))
} \]
and the open mapping theorem.

Consider the family of maps $\pi \circ d_t$ restricted to $W^{n-1}$.  When $t = 0$,
\[ \pi \circ d_0: W^{n-1} \to B^n(\mathcal{C}_0) \] is a topological isomorphism of Banach spaces.  So there is some $\epsilon > 0$ for which $\pi \circ d_t: W^{n-1} \to B^n(\mathcal{C}_0)$ is a topological isomorphism for all $t \in J' := (-\epsilon, \epsilon)$.  From Corollary \ref{Corollary-BanachSpaceInvertibleSmoothFamily}, the induced $C^\infty(J')$-linear map
\[ \pi \circ d: C^\infty(J', W^{n-1}) \to C^\infty(J', B^n(\mathcal{C}_0)) \] is an isomorphism, and in particular it is surjective.  It follows that $\pi: B^n(\mathcal{C}_{J'}) \to C^\infty(J', B^n(\mathcal{C}_0))$ is surjective.

Now consider the map $d_0: W^n \to B^{n+1}(\mathcal{C}_0)$, which is a topological isomorphism of Banach spaces.  In particular it is bounded below, so that
\[ \norm{d_0(w)} \geq C \norm{w}, \qquad \forall w \in W^n \]
for some constant $C >0$.
Since $t \mapsto d_t$ is norm continuous, the maps $d_t: W^n \to \mathcal{C}_t^{n+1}$ are bounded below for $t$ in a small enough interval $J'$.  In particular they are injective.  Let's show $\pi: Z^n(\mathcal{C}_t) \to Z^n(\mathcal{C}_0)$ is injective for all $t \in J'$.  Consider an element $z \in Z^n(\mathcal{C}_t)$.  As vector spaces, $\mathcal{C}_t^n = \mathcal{C}_0^n = Z^n(\mathcal{C}_0) \oplus W^n$, so we can write
\[ z = z_0 + w, \qquad z_0 \in Z^n(\mathcal{C}_0), \quad w \in W^n. \]
If $z \in \ker \pi$, then $z_0 = 0$.  Since $z \in Z^n(\mathcal{C}_t)$, we have $0 = d_t(z) = d_t(w)$.  Since $d_t$ is injective, $w = 0$ and so $z = 0$.  This shows $\pi: Z^n(\mathcal{C}_t) \to Z^n(\mathcal{C}_0)$ is injective for all $t \in J'$, and consequently $\pi: Z^n(\mathcal{C}_{J'}) \to C^\infty(J', Z^n(\mathcal{C}_0))$ is injective.
\end{proof}

A variation of the following theorem was first proved in \cite{MR0634038} using a certain ``inverse function theorem".

\begin{theorem}
Let $A$ be a Banach algebra whose Hochschild cochain complex has a continuous linear contracting homotopy in degree $2$
\[ \xymatrix{
C^1(A, A) \ar@<.5ex>[r]^-{\delta} & C^2(A, A) \ar@<.5ex>[l]^-h \ar@<.5ex>[r]^-{\delta} & C^3(A, A), \ar@<.5ex>[l]^-h
} \qquad \delta h + h \delta = 1,\]
so that $H^2(A,A) = 0$.  Then $A$ is rigid.
\end{theorem}

\begin{proof}
Given a smooth deformation $\{A_t\}_{t \in J}$ with $A_0 = A$, consider the deformation of cochain complexes $\{C^\bullet(A_t, A_t)\}_{t \in J}$.  Using Proposition \ref{Proposition-BanachSpaceSmoothPathsOfHoms}, its complex of sections naturally identifies with the Hochschild complex $C^\bullet_{C^\infty(J)}(A_J, A_J)$.  By Lemma \ref{Lemma-DeformationOfContractibleHasZeroCohomology}, $H^2_{C^\infty(J')}(A_{J'}, A_{J'}) = 0$ for some subinterval $J' \subset J$ containing $0$.  So $A_{J'}$ has a connection that is a derivation, and it is integrable because the underlying space is a Banach space.  This shows $\{A_t\}_{t \in J'}$ is trivial.
\end{proof}

Of course when $A$ is finite dimensional, then such a contracting homotopy will exist whenever $H^2(A,A) = 0$.  As an example, a finite direct sum $A$ of matrix algebras satisfies $H^2(A,A) = 0$.  Thus, all finite dimensional $C^*$-algebras are rigid, as associative algebras.  For a general Banach algebra with $H^2(A,A) = 0$, we will need to assume the existence of the homotopy $h$.  One can possibly circumvent this by considering nonlinear homotopy operators as in \cite{Crainic}.

\begin{theorem} \label{Theorem-ContractibleComplexIsRigid}
Let $(\mathcal{C}^\bullet, d)$ be a contractible cochain complex of Banach spaces that is bounded above and below in degree.  Then the complex $(\mathcal{C}^\bullet, d)$ is rigid.
\end{theorem}

\begin{proof}
Let $h: \mathcal{C}^\bullet \to \mathcal{C}^{\bullet - 1}$ be a continuous contracting homotopy.  Then the endomorphism complex $\End(\mathcal{C})$ is contractible, with homotopy
\[ H: \End(\mathcal{C})^\bullet \to \End(\mathcal{C})^{\bullet-1}, \qquad H(F) = h \circ F. \]  Suppose $\{ (\mathcal{C}_t^\bullet, d_t)\}_{t \in J}$ is a smooth deformation with $\mathcal{C}_0^\bullet = \mathcal{C}^\bullet$.  The complexes $\{ \End(\mathcal{C}_t) \}_{t \in J}$ form a smooth deformation of cochain complexes whose complex of sections is identified with $\End_{C^\infty(J)} (\mathcal{C}_J)$ by Proposition \ref{Proposition-BanachSpaceSmoothPathsOfHoms}.  Our assumption that the degree is bounded guarantees that $\End(\mathcal{C}_t)$ is a Banach space.  By Lemma \ref{Lemma-DeformationOfContractibleHasZeroCohomology}, $H^1(\End_{C^\infty(J')}(\mathcal{C}_{J'})) = 0$ for some subinterval $J'$.  From Proposition \ref{Proposition-ClassOfG}, the module $\mathcal{C}_{J'}^\bullet$ admits a connection that is a chain map.  Since we are working with Banach spaces, the connection is integrable and so $\{ \mathcal{C}_t^\bullet\}_{t \in J'}$ is trivial by Proposition \ref{Proposition-BundleOfComplexesParallelTransport}.
\end{proof}

Next we consider an application to homological perturbation theory.  We recall the construction of the mapping cone.  Given a chain map
\[ f: (\mathcal{C}^\bullet, d_{\mathcal{C}}) \to (\mathcal{D}^\bullet, d_{\mathcal{D}}) \] between cochain complexes, the mapping cone complex $(C_f^\bullet, \partial)$ is defined by
\[ C_{f}^\bullet = \mathcal{C}^{\bullet+1} \oplus \mathcal{D}^\bullet, \qquad \partial = \begin{bmatrix} -d_{\mathcal{C}} & 0\\ f & d_{\mathcal{D}} \end{bmatrix}. \]  If $C_f^\bullet$ is contractible, then it is easy to see that $f$ is a chain homotopy equivalence.  Indeed, one can extract the homotopy inverse as well as the homotopy operators from the contracting homotopy of $C_f^\bullet$.  The converse is true as well \cite{MR1909353}.

\begin{theorem}
Suppose $\{\mathcal{C}_t\}_{t \in J}$ and $\{\mathcal{D}_t\}_{t \in J}$ are smooth deformations of bounded cochain complexes of Banach spaces, and $\{f_t: \mathcal{C}_t \to \mathcal{D}_t\}_{t \in J}$ is a smooth family of continuous chain maps.  If $f_0$ is a chain homotopy equivalence, then there is a subinterval $J' \subset J$ containing $0$ such that $f_t$ is a chain homotopy equivalence for all $t \in J'$.  Moreover, the homotopy inverse and the homotopy operators can be chosen to depend smoothly on $t$.
\end{theorem}

\begin{proof}
By assumption, the mapping cone $C_{f_0}^\bullet$ is contractible, and so it is rigid by Theorem \ref{Theorem-ContractibleComplexIsRigid}.  Thus, the deformation $\{ C_{f_t}^\bullet\}_{t \in J'}$ is trivial for some subinterval $J' \subset J$.  So its complex of sections is isomorphic to $C^\infty(J', C_{f_0})$, which has a $C^\infty(J')$-linear contracting homotopy.  Thus each $C_{f_t}^\bullet$ is contractible in a way that depends smoothly on $t$.
\end{proof}

\section{The Gauss-Manin connection} \label{Section-GaussManinConnection}

\subsection{Gauss-Manin connection in periodic cyclic homology}

In this section, we'll construct Getzler's Gauss-Manin connection in our setting of smooth deformations.  Let $A_J$ denote the algebra of sections of a smooth one-parameter deformation of locally convex algebras $\{A_t\}_{t \in J}$.  Unless specified otherwise, all chain groups and homology groups associated to $A_J$ that follow are over the ground ring $C^\infty(J)$.

Consider the deformation of chain complexes $\{ (C_{\per}(A_t), b_t +B)\}_{t \in J}$.  As in Example \ref{Example-CyclicComplexOfDeformation}, we can identify its complex of sections with $(C_{\per}(A_J), b+B).$  We would like to show, under favorable circumstances, that this deformation of complexes is trivial at the level of homology.  To that end, we'd like to construct a connection on $C_{\per}(A_J)$ that is a chain map.  As described in Proposition \ref{Proposition-ClassOfG}, this is a problem in cohomology.  To start, let $\nabla$ be any connection on $A_J$, and let $E = \delta \nabla$ as in Proposition \ref{Proposition-ClassOfE}.  We extend $\nabla$ to the unitization $(A_J)_+$ (over $C^\infty(J)$) by $\nabla e = 0$, and then to $C_n(A_J)$ using Proposition~\ref{Proposition-TensorAndDualConnections}.  Then $\nabla$ extends to a connection on the periodic cyclic complex $C_{\per}(A_J)$, which is given by the Lie derivative $L_{\nabla}$.  From Proposition \ref{Proposition-ClassOfG}, $C_{\per}(A_J)$ has a connection that is a chain map if and only if the class of
\[ G := [b+B, L_\nabla] = L_E \]
vanishes in $H^1(\End(C_{\per}(A_J)))$, i.e. $L_E$ is chain homotopic to zero via a $C^\infty(J)$-linear homotopy operator.
But the Cartan Homotopy formula
\[ [b+B, I_E] = L_E \]
of Theorem~\ref{Theorem-CHF} implies exactly this.  Notice that $E$ is $C^\infty(J)$-linear, so $I_E$ is a $C^\infty(J)$-linear endomorphism of $C_{\per}(A_J)$.  We conclude that the \emph{Gauss-Manin connection}
\[ \nabla_{GM} = L_\nabla - I_E \]
is a connection on $C_{\per}(A)$ and a chain map.  Amazingly, the cohomological obstruction to the existence of such a connection vanishes for any deformation $\{A_t\}_{t \in J}$.

\begin{proposition}
The Gauss-Manin connection $\nabla_{GM}$ commutes with the differential $b+B$ and hence induces a connection on the $C^{\infty}(J)$-module $HP_{\bullet}(A_J)$.  Moreover, the induced connection on $HP_\bullet(A_J)$ is independent of the choice of connection $\nabla$ on $A_J$.
\end{proposition}

\begin{proof}
We have already established the first claim.  For another connection $\nabla',$ let \[ \nabla'_{GM} = L_{\nabla'} - I_{E'} \] be the corresponding Gauss-Manin connection.  Then \[ \nabla' -\nabla = F, \qquad E' - E = \delta F \]
for some $C^\infty(J)$-linear map $F: A_J \to A_J$.  Thus,
\[ \nabla_{GM}' - \nabla_{GM} = L_F - I_{\delta F} = [b+B, I_F], \]
by Theorem~\ref{Theorem-CHF}.  We conclude that the Gauss-Manin connection is unique up to continuous $C^\infty(J)$-linear chain homotopy.
\end{proof}

\begin{corollary} \label{Corollary-NablaGMOnTrivialDeformation}
If $A$ admits a connection $\nabla$ which is also a derivation, then the Gauss-Manin connection on $HP_{\bullet}(A)$ is given by \[ \nabla_{GM}[\omega] = [L_\nabla \omega]. \]
\end{corollary}

As a trivial example, we see that the Gauss-Manin connection associated to a constant deformation is just the usual differentiation $\frac{d}{dt}$.

The Gauss-Manin connection is a canonical choice of a connection on $HP_\bullet(A_J)$.  It is natural in the sense that morphisms of deformations induce parallel maps at the level of periodic cyclic homology.

\begin{proposition}[Naturality of $\nabla_{GM}$] \label{Proposition-NaturalityOfNablaGM}
Let $A_J$ and $B_J$ denote the algebras of sections of two deformations over the same parameter space $J$, and let $F: A_J \to B_J$ be a morphism of deformations.  Then the following diagram commutes.
\[ \xymatrix{
HP_{\bullet}(A_J) \ar[r]^-{F_*} \ar[d]^-{\nabla_{GM}} &HP_{\bullet}(B_J) \ar[d]^-{\nabla_{GM}}\\
HP_{\bullet}(A_J) \ar[r]^-{F_*}                      &HP_{\bullet}(B_J)\\
} \]
\end{proposition}

\begin{proof}
Let $\nabla^A$ and $\nabla^B$ denote connections on $A_J$ and $B_J$ with respective cocycles $E^A$ and $E^B$, and let $F_*: C_{\per}(A_J) \to C_{\per}(B_J)$ be the induced map of complexes.  For
\[ h = F_* I_{\nabla^A} - I_{\nabla^B} F_*, \] we have
\begin{align*}
[b+B, h] & = F_* [b+B, I_{\nabla^A}] - [b+B, I_{\nabla^B}] F_*\\
&= F_* (L_{\nabla^A} - I_{E^A}) - (L_{\nabla^B} - I_{E^B}) F_*\\
&= F_* \nabla_{GM}^A - \nabla_{GM}^B F_*.
\end{align*} This shows that the diagram commutes up to continuous chain homotopy.  The problem is that $I_{\nabla^A}$ and $I_{\nabla^B}$ are not well-defined operators on the complexes $C_{\per}(A_J)$ and $C_{\per}(B_J)$ respectively (over $C^\infty(J)$), because $\nabla^A$ and $\nabla^B$ are not $C^\infty(J)$-linear operators.  However, one can show that thanks to the Leibniz rule, $h$ descends to a map of quotient complexes such that the following diagram
\[ \xymatrix{
C_{\per}^{\C}(A_J) \ar[r]^-h \ar[d]^-\pi &C_{\per}^{\C}(B_J) \ar[d]^-\pi\\
C_{\per}^{C^{\infty}(J)}(A_J) \ar[r]^-{\bar{h}} &C_{\per}^{C^{\infty}(J)}(B_J)
} \]
commutes, and consequently $[b+B, \bar{h}] = F_* \nabla_{GM}^A - \nabla_{GM}^B F_*$ as desired.
\end{proof}

As a simple application of Proposition~\ref{Proposition-NaturalityOfNablaGM}, we get a proof of the differentiable homotopy invariance property of periodic cyclic homology by considering morphisms between constant deformations.

\begin{corollary}[Homotopy Invariance]
Let $A$ and $B$ be locally convex algebras and let $\{F_t: A \to B\}_{t \in J}$ be a smooth family of algebra maps.  Then the induced map
\[ (F_t)_*: HP_{\bullet}(A) \to HP_{\bullet}(B) \] is independent of $t$.
\end{corollary}

\begin{proof}
Let $A_J = C^{\infty}(J, A)$ and $B_J = C^{\infty}(J, B)$ be the algebras of sections corresponding to the constant deformations over $J$ with fiber $A$ and $B$ respectively.  Then $\{ F_t:A \to B \}_{t \in J}$ is a morphism between these constant deformations.  Let $F: A_J \to B_J$ be the induced $C^\infty(J)$-linear algebra map
\[ F(a)(t) = F_t(a(t)). \]  Using the canonical connection $\frac{d}{dt}$ on both $A_J$, we see that $\nabla_{GM}$ is given by $\frac{d}{dt}$ under the identification $C_{\per}(A_J) \cong C^\infty(J, C_{\per}^{\C}(A))$, and similarly for $B$.  Given a cycle $\omega \in C_{\per}^{\C}(A)$, we view it as a ``constant" cycle in $C_{\per}(A_J),$ and then Proposition~\ref{Proposition-NaturalityOfNablaGM} implies that
\[ \left[ \frac{d}{dt} F_t (\omega) \right] = \left[ F_t \left(\frac{d\omega}{dt}\right) \right] = 0 \]  in $HP(B_J)$.  So there is $\eta \in C^\infty(J,C_{\per}^{\C}(B))$ such that
\[ \frac{d}{dt} F_t(\omega) = (b+B)(\eta(t)).\]  But, by the fundamental theorem of calculus,
\[ F_t(\omega) - F_s(\omega) = \int_s^t(b+B)(\eta(u)) du = (b+B)\Bigg( \int_{s}^{t} \eta(u) du \Bigg) \] for any $s,t \in J$.  Hence $[F_t(\omega)] = [F_s(\omega)]$ in $HP_\bullet(B)$.
\end{proof}

\subsection{Dual Gauss-Manin connection}
We define $\nabla^{GM}$ on $C^{\per}(A_J)$ to be the dual connection of $\nabla_{GM}$ as in Proposition~\ref{Proposition-TensorAndDualConnections}.  In terms of the canonical pairing,
\[ \langle \nabla^{GM}\phi, \omega \rangle = \frac{d}{dt} \langle \phi, \omega \rangle - \langle \phi, \nabla_{GM} \omega \rangle. \]  It is straightforward to verify that $\nabla^{GM}$ commutes with $b+B$ and therefore induces a connection on $HP^\bullet(A_J)$.  The connections $\nabla_{GM}$ and $\nabla^{GM}$ satisfy \[ \frac{d}{dt} \langle [\phi], [\omega] \rangle = \langle \nabla^{GM}[\phi], [\omega] \rangle + \langle [\phi], \nabla_{GM} [\omega] \rangle, \]  for all $[\phi] \in HP^\bullet(A_J)$ and $[\omega] \in HP_{\bullet}(A_J).$

\subsection{Interaction with the Chern character}

The algebra $A_J$ can be viewed as an algebra over $\C$ or $C^\infty(J)$, and there is a surjective morphism of complexes
\[ \pi: C_{\per}^\C(A_J) \to C_{\per}^{C^\infty(J)}(A_J). \]

\begin{proposition}
If $\omega \in C_{\per}^{C^{\infty}(J)}(A_J)$ is a cycle that lifts to a cycle $\tilde{\omega} \in C_{\per}^{\C}(A_J)$, then $\nabla_{GM} [\omega] = 0$ in $HP_{\bullet}^{C^{\infty}(J)}(A_J)$.
\end{proposition}

\begin{proof}
Let $\nabla_{GM}^{\C} = L_{\nabla} - I_E$, viewed as a linear operator on $C_{\per}^{\C}(A_J)$.  By Theorem~\ref{Theorem-CHF}, \[ \nabla_{GM}^\C = L_{\nabla} - I_{\delta \nabla} = [b+B, I_{\nabla}] \] and so $\nabla_{GM}^\C$ is the zero operator on $HP_{\bullet}^{\C}(A_J)$.  Thus, at the level of homology, we have
\[ \nabla_{GM} \circ \pi = \pi \circ \nabla_{GM}^\C = 0 \] where $\pi: HP_\bullet^\C(A_J) \to HP_\bullet^{C^{\infty}(J)}(A_J)$ is the map induced by the quotient map.  By hypothesis, $[\omega]$ is in the image of $\pi$.
\end{proof}

Note that the homotopy used in the previous proof does not imply that $\nabla_{GM}$ is zero on $HP_{\bullet}^{C^{\infty}(J)}(A_J)$.  The reason is that the operator $I_{\nabla}$ is not a well-defined operator on the quotient complex $C_{\per}^{C^{\infty}(J)}(A_J)$.

\begin{theorem} \label{Theorem-ChPChUAreParallel}
If $P \in M_N(A_J)$ is an idempotent and $U \in M_N(A_J)$ is an invertible, then \[ \nabla_{GM}[\ch P] = 0, \qquad \nabla_{GM}[\ch U] = 0 \] in $HP_\bullet(A_J)$.
\end{theorem}

\begin{proof} This is immediate from the previous proposition because the cycle $\ch P \in C_{\per}^\C(A_J)$ is a lift of the cycle $\ch P \in C_{\per}^{C^\infty(J)}(A_J)$, and similarly for $\ch U$.
\end{proof}

Combining this with the identity
\[ \frac{d}{dt} \langle [\phi], [\omega] \rangle = \langle \nabla^{GM}[\phi], [\omega] \rangle + \langle [\phi], \nabla_{GM}[\omega] \rangle, \] we obtain the following differentiation formula for the pairing between $K$-theory and periodic cyclic cohomology.

\begin{corollary}
If $P \in M_N(A_J)$ is an idempotent and $U \in M_N(A_J)$ is an invertible, then
\[ \frac{d}{dt} \langle [\phi], [P] \rangle = \langle \nabla^{GM}[\phi], [P] \rangle, \qquad \frac{d}{dt} \langle [\phi], [U] \rangle = \langle \nabla^{GM}[\phi], [U] \rangle. \]
\end{corollary}

\begin{remark} Proposition~\ref{Proposition-NaturalityOfNablaGM} can be used to give another proof that
\[ \nabla_{GM}[\ch P] = 0 \] when $P \in A_J$ is an idempotent.  Indeed, an idempotent in $A_J$ is equivalent to a morphism of deformations \[ \{ F_t: \C \to A_t\}_{t \in J} \]  from the constant deformation with fiber $\C$.  The induced algebra map
\[ F: C^{\infty}(J, \C) \to A_J \] sends $1$ to $P$.  Applying Proposition~\ref{Proposition-NaturalityOfNablaGM}, we see
\[ \nabla_{GM}[\ch P] = \nabla_{GM} F [\ch 1] = F \frac{d}{dt} [\ch 1] = 0. \]
\end{remark}

\subsection{Integrating $\nabla_{GM}$}

The very fact that $\nabla_{GM}$ exists for all smooth one-parameter deformations implies that the problem of proving $\nabla_{GM}$ is integrable cannot be attacked with methods that are too general.  Indeed, one cannot expect periodic cyclic homology to be rigid for all deformations, there are plenty of finite dimensional examples for which it is not.

\begin{example} \label{Example-CyclicHomologyNotPreserved}
For $t \in \R$, let $A_t$ be the two-dimensional algebra generated by an element $x$ and the unit $1$ subject to the relation $x^2 = t\cdot 1.$  Then $A_t \cong \C \oplus \C$ as an algebra when $t \neq 0$, and $A_0$ is the exterior algebra on a one dimensional vector space.  Consequently,
\[ HP_0(A_t) \cong \begin{cases} \C \oplus \C, &t \neq 0\\ \C, &t = 0. \end{cases} \]  A similar result holds for periodic cyclic cohomology $HP^0(A_t)$.
\end{example}

From the point of view of differential equations, one issue is that the periodic cyclic complex is never a Banach space.  Even in the case where $A$ is a Banach algebra, e.g. finite dimensional, the chain groups $C_n(A)$ are also Banach spaces, but the periodic cyclic complex
\[ C_{\per}(A) = \prod_{n=0}^\infty C_n(A) \] is a Fr\'{e}chet space, as it is a countable product of Banach spaces.  The operator $\nabla_{GM}$ contains the degree $-2$ term $\iota_E: C_n(A) \to C_{n-2}(A)$.  Thus unless $E = 0$, one cannot reduce the problem to the individual Banach space factors, as the differential equations are hopelessly coupled together.

One instance in which $\nabla_{GM}$ is clearly integrable is when the deformation $\{A_t\}_{t \in J}$ is trivial.  Using Proposition \ref{Proposition-IntegrableTensorAndDual} and Proposition \ref{Proposition-TrivialDeformationOfAlgebras}, we obtain the following.

\begin{proposition}
If $A_J$ has an integrable connection $\nabla$ that is a derivation, then $\nabla_{GM} = L_\nabla$ is integrable on $C_{\per}(A_J)$, and
$P^{\nabla_{GM}}_{s,t}: C_{\per}(A_s) \to C_{\per}(A_t)$ is the map of complexes induced by the algebra isomorphism $P^\nabla_{s,t}: A_s \to A_t$.
\end{proposition}

While this is not surprising, it is interesting to note is that if we consider another connection $\nabla'$ on $A_J$, the corresponding Gauss-Manin connection $\nabla_{GM}'$ on $C_{\per}(A_J)$ need not be integrable, and in general seems unlikely to be so.  However the induced connection $\left(\nabla_{GM}'\right)_*$ on $HP_\bullet(A_J)$ is necessarily integrable by the uniqueness of the Gauss-Manin connection up to chain homotopy.

As proving integrability of $\nabla_{GM}$ at the level of the complex $C_{\per}(A_J)$ is both too difficult and, in some cases, too strong of a result, our general approach will be to find a different complex that computes $HP_\bullet(A_J)$ equipped with a compatible connection.

\section{A rigidity theorem for periodic cyclic cohomology}

In this section, we give our main theorem for rigidity of periodic cyclic cohomology of certain Banach algebras.  We first review the necessary concepts from homological algebra.

\subsection{Homological bidimension}

Let $A$ be a (possibly nonunital) Banach algebra, and let $A^e = A_+ \potimes A_+^{\op}$ be its topological enveloping algebra.  The algebra $A^e$ is designed so that there is a one-to-one correspondence between locally convex $A$-bimodules and locally convex unital left $A^e$-modules.  Here, we shall only discuss modules whose underlying space is a Banach space.  The continuous Hochschild cohomology of $A$ with coefficients in a Banach $A$-bimodule $M$ is defined as
\[ H^\bullet(A,M) := \Ext_{A^e}^\bullet(A_+, M). \]  See \cite{MR1093462} for a discussion of derived functors in the context of locally convex algebras and modules.  When $A$ is unital, we do not have to be careful with unitizations, as
\[ H^\bullet(A,M) = \Ext_{A \potimes A^{\op}}^\bullet(A, M). \]  The Hochschild cohomology $H^\bullet(A,A)$ coincides with our previous notation, and $HH^\bullet(A) = H^\bullet(A, A^*)$.  The bimodule structure on $A^*$ comes from a general construction: given any $A$-bimodule $M$, the topological dual $M^* = \Hom(M, \C)$ is an $A$-bimodule via
\[ (a \cdot \phi \cdot b)(m) = \phi(bma), \qquad \forall \phi \in M^*. \]  By considering the topological bar resolution $B_\bullet(A)$, which is a projective resolution of $A_+$ by $A^e$-modules, we obtain the \emph{standard complex}
\[ C^n(A, M) = \Hom_{A^e}(B_n(A), M) \cong \Hom(A^{\potimes n}, M), \] with differential
\begin{align*}
(\delta D)(a_1, \ldots , a_{n+1}) &= a_1 D(a_2, \ldots , a_{n+1}) + (-1)^{n+1}D(a_1, \ldots , a_n)a_{n+1}\\
& \qquad + \sum_{j=1}^n (-1)^j D(a_1, \ldots , a_ja_{j+1}, \ldots , a_{n+1}),
\end{align*} whose cohomology is $H^\bullet(A,M)$, see \cite[Section III.4.2]{MR1093462}.

The \emph{homological bidimension} of a Banach algebra $A$ is
\[ \db A = \inf \{n ~|~ H^{n+1}(A,M) = 0 \text{ for all Banach }A\text{-bimodules }M\} \] and the \emph{weak homological bidimension} of $A$ is
\[ \dbw A = \inf \{n ~|~ H^{n+1}(A,M^*) = 0 \text{ for all Banach }A\text{-bimodules }M\}. \]  Clearly, $\dbw A \leq \db A$.  It is a fact that if $H^{n+1}(A,M) = 0$ (resp. $H^{n+1}(A,M^*) = 0$) for all $M$, then $H^{m}(A,M) = 0$ (resp. $H^m(A,M^*) = 0$) for all $M$ and all $m \geq n+1$ \cite[Theorem III.5.4]{MR1093462} (resp. \cite{MR1437458}).  A Banach algebra $A$ is called \emph{amenable} if $\dbw A = 0$.  As an example, Johnson proved that the convolution algebra $L^1(G)$ of a locally compact group with respect to Haar measure is amenable if and only if the group $G$ is amenable \cite[Theorem 2.5]{MR0374934}.  A Banach algebra $A$ for which $\dbw A = n$ is also called $(n+1)$-amenable.

As in \cite{MR823176}, we shall consider the universal differential graded algebra $(\Omega^\bullet A, d)$ associated to $A$.  However, we shall use the topological version, constructed using completed projective tensor products.  Explicitly, $\Omega^0 A \cong A = C_0(A)$ and
\[ \Omega^n A \cong A_+ \potimes A^{\potimes n} = C_n(A), \] under the identification
\[ a_0da_1 da_2 \ldots da_n \longleftrightarrow (a_0, a_1, a_2, \ldots , a_n). \]  Then $\Omega^n A$ is a Banach $A$-bimodule with the left action
\[ a\cdot(a_0da_1 \ldots da_n) = (aa_0)da_1\ldots da_n.\]  The right action is determined by the relation
\[ (da_1)a_2 = d(a_1a_2) - a_1(da_2). \]  The map
\[ d^{\otimes n}: A^{\potimes n} \to \Omega^n A, \qquad (a_1, \ldots , a_n) \mapsto da_1\cdot \ldots \cdot da_n \] is a Hochschild $n$-cocycle in the standard complex with coefficients in the bimodule $\Omega^nA$.  In fact, $d^{\otimes n}$ is the universal Hochschild $n$-cocycle in the sense that any Hochschild cocycle $D: A^{\potimes n} \to M$ into a Banach $A$-bimodule factors through a unique $A$-bimodule map $F: \Omega^nA \to M$, determined by
\[ F(da_1 \ldots da_n) = D(a_1, \ldots , a_n),\] as in \cite{MR1303030}.  Thus the cohomology class $[D] \in H^n(A,M)$ is the image of $[d^{\otimes n}]$ under the map \[ H^n(A, \Omega^nA) \to H^n(A,M) \] induced by $F$.  It follows that
\[ \db A \leq n \qquad \text{if and only if} \qquad H^{n+1}(A, \Omega^{n+1}A) = 0. \]

Let's now consider cocycles with values in dual Banach modules.  Compose the universal $n$-cocycle $d^{\otimes n}$ with the canonical embedding into the double dual to obtain an $n$-cocycle
\[ d^{\otimes n}: A^{\potimes n} \to (\Omega^nA)^{**}. \]
Given any standard $n$-cocycle $D: A^{\potimes n} \to M^*$, consider the bimodule map
\[ F: \Omega^n A \to M^* \]
induced by the universal property of $\Omega^nA$.  Define a bimodule map
\[ G: M \to (\Omega^nA)^*, \qquad G(m)(\omega) = F(\omega)(m). \]  Then the dual map
\[ G^*: (\Omega^nA)^{**} \to M^* \] is an $A$-bimodule map that satisfies $G^* \circ d^{\otimes n} = D$.
It follows that
\[ \dbw A \leq n \qquad \text{if and only if} \qquad H^{n+1}(A, (\Omega^{n+1}A)^{**}) = 0. \]

Now suppose $\{A_t\}_{t \in J}$ is a smooth deformation of Banach algebras and let $A_J$ be its algebra of sections.  One can form the space of abstract $n$-forms $\Omega^n A_J$ over the ground ring $C^\infty(J)$ by taking the projective tensor products over $C^\infty(J)$.  Notice that the spaces $\{\Omega^nA_t\}_{t \in J}$ are all canonically isomorphic as Banach spaces, and $\Omega^nA_J$ is isomorphic, as a $C^\infty(J)$-module, to the space of smooth functions from $J$ into the underlying Banach space of $\Omega^nA_t$.  There is a universal $C^\infty(J)$-linear cocycle
\[ d^{\otimes n}: A_J^{\potimes_{C^\infty(J)} n} \to \Omega^nA_J, \qquad d^{\otimes n}(a_1, \ldots , a_n) = da_1\ldots da_n. \]  By Propositions \ref{Proposition-TensorProductOfFreeModules} and \ref{Proposition-BanachSpaceSmoothPathsOfHoms}, the complex 
\[ C^\bullet_{C^\infty(J)}(A_J, \Omega^nA_J) = \Hom_{C^\infty(J)}\left( A_J^{\potimes_{C^\infty(J)} \bullet}, \Omega^nA_J \right) \]
is isomorphic to the complex of sections of the deformation $\{C^\bullet(A_t, \Omega^nA_t)\}_{t \in J}$.  Moreover, evaluation at $t \in J$ induces a chain map
\[ \epsilon_t: C^\bullet_{C^\infty(J)}(A_J, \Omega^nA_J) \to C^\bullet(A_t, \Omega^nA_t) \] which maps the universal cocycle for $A_J$ to the universal cocycle for $A_t$.  Thus if $H^{n+1}_{C^\infty(J)}(A_J, \Omega^{n+1}A_J) = 0$, we have $H^{n+1}(A_t, \Omega^{n+1}A_t) = 0$ for all $t \in J$, and consequently $\db A_t \leq n$ for all $t \in J$.

We can also consider the $C^\infty(J)$-linear double dual module $(\Omega^nA_J)^{\dual\dual}$ and the cocycle
\[ d^{\otimes n}: A_J^{\otimes_{C^\infty(J)}n} \to (\Omega^nA_J)^{\dual\dual} \] obtained by composing the universal cocycle with the canonical embedding into the double dual.  Using Proposition \ref{Proposition-BanachSpaceSmoothPathsOfHoms} and Corollary \ref{Corollary-DualFreeModule}, the Hochschild complex $C^\bullet_{C^\infty(J)}(A_J, (\Omega^nA_J)^{\dual\dual})$ identifies with the complex of sections of the deformation $\{C^\bullet(A_t, (\Omega^nA_t)^{**})\}_{t \in J}$.  By considering evaluation at $t \in J$, we see that if $[d^{\otimes n}] = 0$ in $H^n_{C^\infty(J)}(A_J, (\Omega^nA_J)^{\dual\dual})$, then $[d^{\otimes n}] = 0$ in $H^n(A_t, (\Omega^nA_t)^**)$ for all $t \in J$.  We have proved the following proposition.

\begin{proposition}
Let $\{A_t\}_{t \in J}$ be a deformation of Banach algebras.
\begin{enumerate}
\item If $H^{n+1}_{C^\infty(J)}(A_J, \Omega^nA_J) = 0$, then $\db A_t \leq n$ for all $t \in J$.
\item If $H^{n+1}_{C^\infty(J)}(A_J, (\Omega^nA_J)^{\dual\dual}) = 0$, then $\dbw A_t \leq n$ for all $t \in J$.
\end{enumerate}
\end{proposition}

\begin{lemma} \label{Lemma-FiniteHochschildDimensionContraction}
Let $A$ be a Banach algebra and $M$ be a Banach $A$-bimodule.
\begin{enumerate}
\item If $\db A \leq n$, then the standard complex $C^\bullet(A,M)$ has a contracting homotopy in degree $n+1$
\[ \xymatrix{
C^n(A, M) \ar@<.5ex>[r]^-{\delta} & C^{n+1}(A, M) \ar@<.5ex>[l]^-h \ar@<.5ex>[r]^-{\delta} & C^{n+2}(A, M), \ar@<.5ex>[l]^-h
} \qquad \delta h + h \delta = 1.\]
\item If $\dbw A \leq n$, then the standard complex $C^\bullet(A,M^*)$ has a contracting homotopy in degree $n+1$
\[ \xymatrix{
C^n(A, M^*) \ar@<.5ex>[r]^-{\delta} & C^{n+1}(A, M^*) \ar@<.5ex>[l]^-h \ar@<.5ex>[r]^-{\delta} & C^{n+2}(A, M^*), \ar@<.5ex>[l]^-h
} \qquad \delta h + h \delta = 1.\]
\end{enumerate}
\end{lemma}

\begin{proof}
If $\db A \leq n$, then $A_+$ has a projective resolution of length $n$ \cite[Theorem III.5.4]{MR1093462}.  By the ``Comparison theorem" \cite[Theorem III.2.3]{MR1093462}, a projective resolution is unique up to chain homotopy equivalence in the category of complexes of Banach $A$-bimodules.  By applying the functor $\Hom_{A^e}(\cdot, M)$, it follows that the standard complex $C^\bullet(A,M)$ has the required homotopy.

If $\dbw A \leq n$, then $A_+$ has a flat resolution of length $n$ \cite[Theorem 1]{MR1437458}, that is, a resolution by Banach $A$-bimodules which are flat as left $A^e$-modules.  Since the dual of a flat module is injective \cite[Theorem VII.1.14]{MR1093462}, it follows that $A^*$ has an injective resolution of length $n$.  Using the Comparison theorem for injective resolutions, the dual $B_\bullet(A)^*$ of the bar resolution has a contracting homotopy in degree $n+1$ consisting of $A$-bimodule maps.  After applying the functor $\Hom_{A^e}(M, \cdot)$, we see the complex $\Hom_{A^e}(M, B_\bullet(A)^*)$ has a contracting homotopy in degree $n+1$.  However there is a natural isomorphism of complexes
\[ \Hom_{A^e}(M, B_\bullet(A)^*) \cong \Hom_{A^e}(B_\bullet(A), M^*) = C^\bullet(A, M^*),\] which gives the result, see \cite[Proposition III.4.13]{MR1093462}.
\end{proof}

\begin{corollary} \label{Corollary-UpperSemicontinuity}
Let $\{A_t\}_{t \in J}$ be a smooth deformation of Banach algebras.  Then the functions
\[ t \mapsto \db A_t \qquad \text{and} \qquad t \mapsto \dbw A_t \] are upper semi-continuous.
\end{corollary}

\begin{proof}
%
Simply combine the previous two results with Lemma \ref{Lemma-DeformationOfContractibleHasZeroCohomology}.
\end{proof}

\begin{example}
Let $G$ be a connected semisimple Lie group with maximal compact subgroup $K$.  Let $\mathfrak{g}$ and $\mathfrak{k}$ denote their respective Lie algebras.  In \cite{MR2391803}, a smooth deformation $\{G_t\}$ of Lie groups is constructed in such a way that $G_0 = \mathfrak{g}/\mathfrak{k} \rtimes K$ and $G_t \cong G$ for all $t \neq 0$.  The group $G_0$ is amenable, but $G$ may not be, e.g. $G = SL(2,\R)$.  So Corollary \ref{Corollary-UpperSemicontinuity} and Johnson's theorem imply that there is no corresponding smooth deformation $\{L^1(G_t)\}$ of Banach algebras.
\end{example}


\subsection{Contractions and retractions}
When the universal $(n+1)$-cocycle is a coboundary, one can construct a contracting homotopy in the Hochschild complex in a uniform way.  As described in \cite{MR1269386}, if $\phi: A^{\potimes n} \to \Omega^{n+1}A$ satisfies $\delta \phi = d^{\otimes (n+1)}$, then
\[ \alpha: \Omega^k A \to \Omega^{k+1}A, \qquad \alpha(a_0 da_1 \ldots da_k) = a_0 \phi(a_1, \ldots , a_n) da_{n+1} \ldots da_k \]
defines a contracting homotopy of the Hochschild chain complex $(C_\bullet(A), b)$ in degrees $k \geq n+1$.  The transpose of $\alpha$ gives a contracting homotopy in degree $k \geq n+1$ for the Hochschild cochain complex $(C^\bullet(A), b)$.

Khalkhali showed in \cite{MR1269386} that if the cocycle $d^{\otimes (n+1)}: A^{\potimes (n+1)} \to (\Omega^{n+1}A)^{**}$ is a coboundary, then one can construct a contracting homotopy of $(C^\bullet(A), b)$ in degrees $k \geq n+1$ in a similar way.  Given a cochain $\phi: A^{\potimes n} \to (\Omega^{n+1}A)^{**}$ such that $\delta \phi = d^{\otimes (n+1)}$, define
\[ \alpha: (\Omega^{k+1}A)^* \to (\Omega^kA)^*, \qquad k \geq n \] by
\[ (\alpha f)(a_0 da_1 \ldots da_k) = \left[ a_0 \cdot \phi(a_1, \ldots , a_n) \right](f_{da_{n+1} \ldots da_k}), \]
where $f \in (\Omega^{k+1}A)^*$ and $f_{da_{n+1} \ldots da_k} \in (\Omega^{n+1}A)^*$ is given by
\[ f_{da_{n+1} \ldots da_k}(\omega) = f(\omega da_{n+1} \ldots da_k). \]
Then $b\alpha + \alpha b = 1$ in $C^k(A)$ when $k \geq n+1$.

Given a contracting homotopy $\alpha: C^{k+1}(A) \to C^k(A)$, Khalkhali constructed a retract of the periodic cyclic cochain complex with only finitely many degrees \cite{MR1269386}, which we now describe.  Let $N$ be such that $\alpha$ is a contracting homotopy in all degrees above $2N$.  Let
\[ C_0^{\even}(A) = \left(\bigoplus_{k=0}^{N-1} C^{2k}(A)\right) \bigoplus \ker \left\{ b: C^{2N}(A) \to C^{2N+1}(A) \right\} \] and
\[ C_0^{\odd}(A) = \bigoplus_{k=0}^{N-1} C^{2k+1}(A). \]
The $\Z/2$-graded complex $C_0^{\per}(A) = C_0^{\even}(A) \oplus C_0^{\odd}(A)$ has differential $b+B$.  Then $C_0^{\per}(A)$ is a subcomplex of $C^{\per}(A)$, and in fact is a deformation retract.  That is, there is a chain map $R: C^{\per}(A) \to C_0^{\per}(A)$ such that $RI = \id$ and $IR$ is chain homotopic to $\id$, where $I: C_0^{\per}(A) \to C^{\per}(A)$ is the inclusion.  Thus, the cohomology of $C_0^{\per}(A)$ is $HP^\bullet(A)$.  The key feature is that $C_0^{\per}(A)$ is a complex of Banach spaces.  We won't need the explicit form of the retraction $R$, but we remark that it depends heavily on the contracting homotopy $\alpha$.

All of the above can be carried out for the algebra of sections $A_J$ of a smooth deformation $\{A_t\}_{t \in J}$ of Banach algebras, where everything is considered over the ground ring $C^\infty(J)$.  If $\phi: A_J^{\potimes n} \to \Omega^{n+1}A_J$ satisfies $\delta \phi = d^{\otimes (n+1)}$, then
\[ \alpha: \Omega^k A_J \to \Omega^{k+1}A_J, \qquad \alpha(a_0 da_1 \ldots da_k) = a_0 \phi(a_1, \ldots , a_n) da_{n+1} \ldots da_k \]
defines a contracting homotopy of the Hochschild chain complex $(C_\bullet(A_J), b)$ in degrees $k \geq n+1$.  Its dual
\[ \alpha^\dual: C^{k+1}(A_J) \to C^k(A_J) \] is a contraction for the Hochschild cochain complex.

If $\phi: A_J^{\otimes n} \to (\Omega^{n+1}A_J)^{\dual \dual}$, then Khalkhali's contracting homotopy
\[ \alpha: (\Omega^{k+1}A_J)^\dual \to (\Omega^kA_J)^\dual, \qquad k \geq n\] can be defined by the same formula as above
\[ (\alpha f)(a_0 da_1 \ldots da_k) = \left[ a_0 \cdot \phi(a_1, \ldots , a_n) \right](f_{da_{n+1} \ldots da_k}), \] where $f \in (\Omega^{k+1}A_J)^\dual$ and $f_{da_{n+1} \ldots da_k} \in (\Omega^{n+1}A_J)^\dual$ is given by
\[ f_{da_{n+1} \ldots da_k}(\omega) = f(\omega da_{n+1} \ldots da_k). \]  So $b\alpha + \alpha b = 1$ in $C^k(A_J)$ for $k \geq n+1$.  Moreover, given an $\alpha$ which is a contracting homotopy in degrees above $2N$, we can define
\[ C_0^{\even}(A_J) = \left(\bigoplus_{k=0}^{N-1} C^{2k}(A_J)\right) \bigoplus \ker \left\{ b: C^{2N}(A_J) \to C^{2N+1}(A_J) \right\} \] and
\[ C_0^{\odd}(A_J) = \bigoplus_{k=0}^{N-1} C^{2k+1}(A_J). \]
As in the $\C$-linear case, there is an inclusion $I: C_0^{\per}(A_J) \to C^{\per}(A_J)$ and a retraction $R: C^{\per}(A_J) \to C_0^{\per}(A_J)$ which are $C^\infty(J)$-linear chain maps such that $RI = \id$ and $IR$ is chain homotopic to $\id$.  The retraction $R$ is built in the same way as the $\C$-linear case using the homotopy $\alpha$.

\begin{definition}
We'll say that a locally convex algebra $A$ is $HP^\bullet$-rigid if whenever $\{A_t\}_{t \in J}$ is a smooth deformation with $A_0 = A$, then there some subinterval $J' \subseteq J$ containing $0$ for which $HP^\bullet(A_t) \cong HP^\bullet(A_0)$ for all $t \in J'$.
\end{definition}

We now give our main application of the Gauss-Manin connection.

\begin{theorem} \label{Theorem-HPRigidity}
Let $A$ be a Banach algebra such that $\dbw A < \infty$.  Then $A$ is $HP^\bullet$-rigid.
\end{theorem}


\begin{proof}
Let $\{A_t\}_{t \in J}$ be a smooth deformation with $A_0 = A$.  Suppose $\dbw A_0 = n$.  As described above, the Hochschild complex $C^{\bullet}_{C^\infty(J)}(A_J, (\Omega^{n+1}A_J)^{\dual \dual})$ identifies with the complex of sections of the deformation $\{C^\bullet(A_t, (\Omega^{n+1}A_t)^{**})\}_{t\in J}$.  From Lemmas \ref{Lemma-FiniteHochschildDimensionContraction} and \ref{Lemma-DeformationOfContractibleHasZeroCohomology}, $H^{n+1}_{C^\infty(J')}(A_{J'}, (\Omega^{n+1}A_{J'})^{\dual \dual}) = 0$ for a subinterval $J' \subseteq J$ containing $0$.  So there is a $\phi: A_{J'}^{\potimes n} \to (\Omega^{n+1}A_{J'})^{\dual \dual}$ with $\delta \phi = d^{\otimes (n+1)}$.  As described above, we can construct from this the deformation retract $C_0^{\per}(A_{J'})$ of $C^{\per}(A_{J'})$ for a suitable $N$.  A priori, the space of cocycles $\ker \{ b: C^{2N}(A_{J'}) \to C^{2N+1}(A_{J'}) \}$ may not be a free $C^\infty(J')$-module.  However, the conclusion of Lemma \ref{Lemma-DeformationOfContractibleHasZeroCohomology} guarantees that it is, and moreover $C_0^{\per}(A_{J'})$ is the complex of sections of $\{C_0^{\per}(A_t)\}_{t \in J'}$.

We can now transfer the Gauss-Manin connection to $C_0^{\per}(A_{J'})$.  Let
\[ I: C_0^{\per}(A_{J'}) \to C^{\per}(A_{J'}), \qquad R: C^{\per}(A_{J'}) \to C_0^{\per}(A_{J'}) \] be the inclusion and retraction, which are continuous $C^\infty(J')$-linear chain maps.  Define $\widetilde{\nabla} = R \circ \nabla_{GM} \circ I$ on $C_0^{\per}(A_{J'})$.  Then $\widetilde{\nabla}$ is a chain map and it is a connection because $RI = \id$.  Since the underlying space $C_0^{\per}(A_t)$ is a Banach space, the connection $\widetilde{\nabla}$ is integrable, and the result follows from Proposition \ref{Proposition-BundleOfComplexesParallelTransport}.
\end{proof}

Let $HE^\bullet(A)$ denote the entire cyclic cohomology of $A$, see \cite{MR953915}.  As Khalkhali showed, the canonical inclusion $HP^\bullet(A) \to HE^\bullet(A)$ is an isomorphism for Banach algebras of finite weak bidimension \cite{MR1269386}.  We immediately obtain the following.

\begin{corollary}
Let $A$ be a Banach algebra such that $\dbw A < \infty$.  Then $A$ is $HE^\bullet$-rigid.
\end{corollary}

\begin{example}
We'll show how our theorem can be used to give a proof of an instance of a theorem of Block on the cyclic homology of filtered algebras \cite{MR934456}.  In Block's setting we have an increasing filtration of an algebra $A$,
\[ F_0 \subset F_1 \subset F_2 \subset \ldots \]
where $A = \bigcup_n F_n$ and $F_n\cdot F_m \subset F_{n+m}$.  Letting $B = \gr(A)$ be the associated graded algebra, his result is that if $HH_n(A) = 0$ for all large enough $n$, then the inclusion $F_0 \to A$ induces an isomorphism $HP_\bullet(F_0) \cong HP_\bullet(A)$.
Let's consider the situation of a finite filtration of a Banach algebra
\[ F_0 \subset F_1 \subset \ldots \subset F_N = A, \] and suppose there exist closed subspaces $B_k \subset A$ for which each $F_n \cong \bigoplus_{k=0}^n B_k$ as Banach spaces.  Then we can identify the associated graded algebra $\gr(A) \cong \bigoplus_{k=0}^N B_k$ with $A$ as Banach spaces.  The multiplication in $\gr(A)$ is such that $B_n\cdot B_m \subset B_{n+m}$.  Given $a \in B_n$ and $b \in B_m$, the product in the filtered algebra $A$ can be written as
\[ ab = \sum_{k=0}^{n+m} \pi_{n,m}^k(a,b) \] for uniquely defined operators
\[ \pi_{n,m}^k: B_n \potimes B_m \to B_{n+m-k}. \]  Given $t \in \R$, we can define a new associative product $m_t$ on $A$ by
\[ m_t(a,b) = \sum_{k=0}^{n+m} t^k \pi_{n,m}^k(a,b), \qquad a \in B_n, b \in B_m. \]  This clearly gives a smooth deformation $\{A_t\}_{t \in \R}$ of Banach algebras, as the products depend polynomially on $t$.  We have $A_1 = A$, $A_0 = \gr(A)$, and $A_t \cong A$ for all $t \neq 0$.  If $\dbw \gr(A) < \infty$, then the Gauss-Manin connection is integrable for this deformation, and $HP^\bullet(A) \cong HP^\bullet(\gr(A))$.  View the inclusions $\{F_0 \to A_t\}_{t \in J}$ as a morphism of deformations out of the constant deformation.  From Proposition \ref{Proposition-NaturalityOfNablaGM}, this morphism induces a $\nabla^{GM}$-parallel map.  Since $HP^\bullet(\gr(A)) \to HP^\bullet(F_0)$ is an isomorphism (Example \ref{Example-GradedAlgebraCyclicHomology}), it follows that $HP^\bullet(A) \to HP^\bullet(F_0)$ is an isomorphism.
\end{example}

Notice that Example \ref{Example-CyclicHomologyNotPreserved} is such a deformation of a filtered algebra $A_1$ into its associated graded algebra $A_0$.  However $\dbw A_0 = \infty$, as one can show that $HH^n(A_0) \cong \C$ for all $n > 0$.


\end{document}